\documentclass[11pt]{article}
\usepackage[colorlinks=true, urlcolor=blue, linkcolor=red]{hyperref}
\usepackage{authblk}
\usepackage[utf8]{inputenc}
\usepackage{mathtools}
\usepackage{latexsym}
\usepackage{amsmath}
\usepackage{cleveref}
\usepackage{nicematrix}

\usepackage{stmaryrd}
\usepackage{amssymb}
\usepackage{amsthm}
\usepackage{epsfig}
\usepackage{color}
\usepackage{hyperref}
\usepackage{multicol}
\usepackage{pstricks}
\usepackage{pst-node}
\usepackage{pst-tree}
\usepackage{pst-coil}
\usepackage{multirow}
\usepackage{graphics}
\usepackage{tikz}
\usetikzlibrary{automata, positioning, arrows}
\usetikzlibrary {graphs.standard}

\makeatletter
\DeclareRobustCommand\bigop[1]{%
  \mathop{\vphantom{\sum}\mathpalette\bigop@{#1}}\slimits@
}
\newcommand{\bigop@}[2]{%
  \vcenter{%
    \sbox\z@{$#1\sum$}%
    \hbox{\resizebox{\ifx#1\displaystyle.9\fi\dimexpr\ht\z@+\dp\z@}{!}{$\m@th#2$}}%
  }%
}
\makeatother

\theoremstyle{definition}
\newtheorem{theorem}{Theorem}

\newtheorem{lemma}[theorem]{Lemma}

\newtheorem{proposition}[theorem]{Proposition}
\newtheorem{definition}[theorem]{Definition}

\advance \topmargin by -\headheight
\advance \topmargin by -\headsep
\evensidemargin \oddsidemargin
\begin{document}


\title{Transcendence for Pisot Morphic Words over an Algebraic Base}

\author[1]{Pavol Kebis}
\author[2]{Florian Luca}
\author[3]{Joel Ouaknine}
\author[4]{Andrew Scoones}
\author[4]{James Worrell}

\affil[1]{Institute of Science and Technology Austria,  Klosterneuburg, Austria
\thanks{\texttt{pavol.kebis@ist.ac.at}}}

\affil[2]{Department of Mathematical Sciences, Stellenbosch University, South Africa
\thanks{\texttt{fluca@sun.ac.za}}}

\affil[3]{Max Planck Institute for Software Systems, Saarbrücken, Germany
\thanks{\texttt{joel@mpi-sws.org}}}

\affil[4]{Department of Computer Science, University of Oxford, UK
\thanks{\texttt{\{andrew.scoones,james.worrell\}@cs.ox.ac.uk}}}

\date{}
\maketitle

\begin{abstract}
  It is known that for a uniform morphic sequence
  $\boldsymbol u = \langle u_n\rangle_{n=0}^\infty$ and an algebraic
  number $\beta$ such that $|\beta|>1$, the number
  $\llbracket\boldsymbol{u} \rrbracket_\beta:=\sum_{n=0}^\infty
  \frac{u_n}{\beta^n}$ either lies in $\mathbb Q(\beta)$ or is
  transcendental.  In this paper we show a similar
  rational-transcendental dichotomy for sequences defined by
  irreducible Pisot morphisms.  Subject to the Pisot conjecture (an
  irreducible Pisot morphism has pure discrete spectrum), we
  generalise the latter result to arbitrary finite alphabets.  In
  certain cases we are able to show transcendence of
  $\llbracket\boldsymbol{u}\rrbracket_{\beta}$ outright.  In particular, for
  $k\geq 2$, if $\boldsymbol u$ is the $k$-bonacci word then
  $\llbracket\boldsymbol{u}\rrbracket_{\beta}$ is transcendental.
      \end{abstract}

\section{Introduction}
In 1968 Cobham~\cite{Cob2} conjectured that every number whose expansion
  in an integer base is an automatic sequence is either
  rational or transcendental~\cite{Cob2}.  Recall here that a sequence
  $\boldsymbol u=\langle u_n\rangle_{n=0}^{\infty}$ is $k$-automatic for
  $k \geq 2$ if there is automaton that outputs $u_n$ when given a
  base-$k$ representation of $n$ as input.  A $k$-automatic sequence
  can equivalently be characterised as the fixed point of a
  $k$-uniform morphism, that is, a morphism that maps each letter to
  string of length $k$.  The transcendence of irrational automatic
  numbers over an integer base was proven in 2004 by Adamczewski,
  Bugeaud, and Luca~\cite{ABL}, using Schlickewei's $p$-adic Subspace
  Theorem.
  
In~\cite{Cob2} Cobham also stated a more general version of his
  conjecture, in which the rational-transcendental dichotomy is
  posited for numbers whose expansion in an integer base is given by a
  morphism of exponential growth.  Recall here that a primitive
  morphism has exponential growth if its expansion factor (the
  Perron-Frobenius eigenvalue of its characteristic matrix) is
  strictly greater than one.  The expansion factor of a $k$-uniform
  morphism is $k$ and hence such morphisms have exponential growth for
  $k\geq 2$.  This second form of Cobham's conjecture was confirmed
  in~\cite{ACG} using a construction similar to that used in the case
  of uniform morphisms.  It is noted in~\cite{AB1} that the further
  extension to words generated by morphisms of merely polynomial growth
  encompasses several other longstanding open problems in
  transcendence theory.  

  Another natural generalisation of Cobham's conjecture involves a
  formulation for algebraic-number as opposed to integer bases:
  for $b\geq 2$,
  $\boldsymbol{u} \in \{0,1,\ldots,b-1\}^\omega$ a fixed point of a morphism of
  exponential growth, and $\beta$ an algebraic number, $|\beta|>1$,
  the number
  $\llbracket\boldsymbol{u}\rrbracket_{\beta}:=\sum_{n=0}^\infty
  \frac{u_n}{\beta^n}$ either lies in $\mathbb Q(\beta)$ or is
  transcendental.  Such a result was recently obtained in the case of
  uniform morphisms by Adamczewski and Faverjon~\cite{AF} using
  Mahler's method.  Beyond the case of uniform morphisms,
  transcendence results for morphic words over algebraic bases are
  more restricted.  However, the techniques of~\cite{AB1} can be
  extended to prove that for a word $\boldsymbol u$ and algebraic
  number $\beta$, subject to a non-trivial inequality between the
  height of $\beta$ and a measure of the periodicity of
  $\boldsymbol u$, called the \emph{Diophantine exponent}, the number
  $\llbracket\boldsymbol{u}\rrbracket_{\beta}$ is either
  transcendental or lies in $\mathbb{Q}(\beta)$; see~\cite[Theorem
  1]{AB2}.  This result can be applied in the case that
  $\boldsymbol u$ is the fixed point of a morphism of exponential
  growth and $\beta$ has height bounded above by a quantity determined
  by $\boldsymbol u$.
  
  The object of this paper is to make further progress on the
  extension of Cobham's conjecture to algebraic-number bases.  We
  focus on morphisms whose expansion factor is a \emph{Pisot number},
  that is, an algebraic integer whose Galois conjugates all have
  absolute value strictly less than one.  Pisot morphisms are of
  particular interest in view of the \emph{Pisot Substitution
    Conjecture}.  This says that the shift dynamical system associated
  with an irreducible Pisot morphism has pure point spectrum
  (see~\cite{Akiyama2015,host1989spectral} and~\cite[Chapter
  6.3]{Queffelec2010}) or, equivalently, is measurably isomorphic to a
  translation on a compact abelian group equipped with the Haar
  measure~\cite[Theorem 9.3]{Glasner2003}.  Notable positive examples
  of this conjecture include the Fibonacci morphism, whose associated
  shift dynamical system arises as the coding of a rotation on the
  unit circle, and the Tribonacci morphism, whose associated dynamical
  system arises as the coding, via the Rauzy fractal, of a translation
  on the two-torus~\cite{AR91,Rauzy1982}.
  
  One of our main contributions is to exhibit a link between the
  respective conjectures of Pisot and Cobham.  This is established via
  a combinatorial criterion on an infinite word called \emph{echoing},
  given in Definition~\ref{def:echoing}.  In Theorem~\ref{thm:main} we
  show that if $\boldsymbol u \in \{0,\ldots,b-1\}^{\omega}$ is
  echoing then for all algebraic numbers $\beta$, either
  $\llbracket\boldsymbol{u}\rrbracket_{\beta}$ is transcendental or it lies in the
  field $\mathbb Q(\beta)$.  We furthermore give a sufficient
  condition for a word to be echoing in terms of Livshits's
  balanced-pair algorithm: namely we show that a morphic word on which
  this algorithm terminates with coincidence is echoing.  It is known
  that the algorithm terminates with coincidence for any irreducible
  Pisot morphism on a binary alphabet.  Moreover, for an irreducible
  Pisot morphism on an arbitrary alphabet, the algorithm terminates
  with coincidence if and only if the morphism has pure discrete
  spectrum~\cite[Theorem 5.3]{Akiyama2015}.  We conclude that all
  words defined by irreducible Pisot morphisms on a binary alphabet
  are echoing and hence obey (the algebraic-number-base extension of)
  Cobham's conjecture.  Assuming the Pisot conjecture, words defined
  by irreducible Pisot morphisms on arbitrary finite alphabets are
  also echoing and thereby obey Cobham's conjecture.

  We further propose in Definition~\ref{def:echoing2} a strengthening
  of the echoing condition, called \emph{strongly echoing}, and show
  that if $\boldsymbol u \in \{0,\ldots,b-1\}^{\omega}$ is strongly
  echoing then $\llbracket\boldsymbol{u}\rrbracket_{\beta}$ is transcendental, i.e., we
  eliminate the eventuality that $\llbracket\boldsymbol{u}\rrbracket_{\beta}$ lies in
  $\mathbb Q(\beta)$.  By a combinatorial analysis of balanced pairs,
  we show that for all $k\geq 2$ the $k$-bonacci word is strongly
  echoing and hence that for any such word $\boldsymbol u$ the number
  $\llbracket\boldsymbol{u}\rrbracket_{\beta}$ is transcendental.  
  
  The present paper is a development of the conference
  paper~\cite{KebisLOS024}.  The combinatorial transcendence
  conditions presented below are based on a condition introduced
  therein and in the MSc thesis~\cite{PK}.  These conditions in turn
  build on ideas introduced in~\cite{LOW} to prove transcendence
  results for Sturmian sequences.  In the present paper we have
  divided the notion of echoing sequence into weak and strong
  variants, highlighting the importance of the non-vanishing condition
  in the formulation of the latter.  The applications to Pisot
  morphisms in Section~\ref{sec:binary}, including the connection with
  the Pisot conjecture, and $k$-bonacci sequences in
  Section~\ref{sec:kbon} did not appear in~\cite{KebisLOS024}.

\section{Preliminaries}
\label{sec:prelim}

\subsection{Morphic Sequences}
Consider an alphabet $\Sigma = \{0,\ldots,k-1\}$.
We endow the set
$\Sigma^{\omega}$ of infinite words over $\Sigma$ with the product topology, where $\Sigma$ has the
discrete topology.  The \emph{shift map} $\sigma :\Sigma^{\omega}\rightarrow \Sigma^{\omega}$ is defined by
$\sigma(u_0u_1u_2 \cdots) = u_1u_2u_3\cdots$.

A \emph{morphism} is a homomorphism
$\varphi : \Sigma^+ \rightarrow \Sigma^+$ of the free semigroup
$\Sigma^+$.  The \emph{incidence matrix}
$M_\varphi \in \mathbb N^{k\times k}$ of $\varphi$ is defined by
taking $(M_\varphi)_{i,j}$ to be the number of occurrences of the
symbol $j$ in $\varphi(i)$.  We say that $\varphi$ is \emph{primitive}
if some power of $M_\varphi$ is positive, and we say that $\varphi$
has \emph{exponential growth} if the spectral radius
$\rho(M_{\varphi})$ of $M_\varphi$ is strictly greater than one.  If
$M_\varphi$ is primitive, then by the Perron-Frobenius theorem its
spectral radius is a strictly positive real eigenvalue and all other
eigenvalues have absolute value strictly less than
$\rho(M_{\varphi})$.  We say that $\varphi$ is of \emph{Pisot type}
(or simply Pisot) if $\rho(M_{\varphi})$ is a Pisot number, that is,
all its Galois conjugates have absolute value strictly less than one.

Let $\varphi$ be a primitive morphism over alphabet $\Sigma$.  Assume
that $\varphi$ is \emph{prolongable}, that is, $\varphi(0) = 0u$ for
some word $u \in \Sigma^+$.  Then the sequence
$(\varphi^n(0))_{n\geq 0}$ of finite words converges to an infinite
word $\boldsymbol u$ that is a fixed point of $\varphi$.  The word
$\boldsymbol u$ is \emph{uniformly recurrent}, that is, each finite
factor occurs infinitely often and with bounded gaps between
successive occurrences.

We refer to the shift
dynamical system $(X_\varphi,\sigma)$, where $X_\varphi$ is the
topological closure of the orbit of $\boldsymbol u$ under the shift
map and, by a slight abuse of notation, $\sigma$ denotes the
restriction of the shift map to $X_\varphi$.  For a primitive morphism
this dynamical system has a unique ergodic measure $\mu$ and we have
a unitary operator
$U_\varphi : L^2(X_\varphi,\mu) \rightarrow L^2(X_\varphi,\mu)$ given
by $f\mapsto f\circ \sigma$.
The morphism $\varphi$ is said to have \emph{pure discrete spectrum} if the space
$L^2(X_\varphi,\mu)$ has a basis of eigenvectors of $U_\varphi$.

\subsection{Number Theory}
Let $K$ be a number field of degree $d$ over $\mathbb Q$ and let
$M(K)$ be the set of \emph{places} of $K$.  We divide $M(K)$ into the
collection of \emph{Archimedean places}, which are determined either
by an embedding of $K$ in $\mathbb{R}$ or a complex-conjugate pair of
embeddings of $K$ in $\mathbb{C}$, and the set of
\emph{non-Archimedean places}, which are determined by prime ideals in
the ring $\mathcal{O}_K$ of integers of $K$.

For $a \in K$ and $v \in M(K)$, define the absolute value $|a|_v$ as
follows: $|a|_v := |\sigma(a)|^{1/d}$ if $v$ corresponds to a real
embedding $\sigma:K\rightarrow \mathbb{R}$;
$|a|_v := |\sigma(a)|^{2/d}$ if $v$ corresponds to a complex-conjugate
pair of embeddings
$\sigma,\overline{\sigma}:K \rightarrow \mathbb{C}$; 
$|a|_v := N(\mathfrak{p})^{-\mathrm{ord}_{\mathfrak{p}}(a)/d}$ if $v$
corresponds to a prime ideal $\mathfrak{p}$ in $\mathcal{O}$ and
$\mathrm{ord}_{\mathfrak{p}}(a)$ is the order to which $\mathfrak{p}$
divides the ideal $a\mathcal{O}$.  With the above definitions we have
the \emph{product formula}: $\prod_{v \in M(K)} |a|_v = 1$ for all
$a \in K^\times$.  Given a set of places $S\subseteq M(K)$, the ring
$\mathcal{O}_S$ of \emph{$S$-integers} is the subring comprising all
$a \in K$ such $|a|_v \leq 1$ for all non-Archimedean places
$v\not\in S$.

For $m\geq 1$ the \emph{Weil height} of the projective point
$\boldsymbol{a}=[a_0 : a_1 : \cdots : a_m] \in \mathbb{P}^m(K)$ is 
\[ H(\boldsymbol{a}):=\prod_{v \in M(K)}\max(|a_0|_v,\ldots,|a_m|_v)
  \, .\] This definition is independent of the choice of the field $K$
containing $a_0,\ldots,a_m$.  We define the height $H(a)$ of $a \in K$
to be the height $H([1 : a])$ of the corresponding point in
$\mathbb{P}^1(K)$.  For a non-zero Laurent polynomial
$f = x^n \sum_{i=0}^m a_i x^i \in K[x,x^{-1}]$, where $m\geq 1$ and
$n\in \mathbb Z$, following~\cite{LEN97} we define its height $H(f)$
to be the height $H([a_0 : \cdots : a_m])$ of the vector of
coefficients.

The following special case\footnote{We formulate the special case of
  the Subspace Theorem in which all but one of the linear forms are
  coordinate variables.} of the $p$-adic Subspace Theorem of
Schlickewei~\cite{Schlickewei76} is one of the main ingredients of our
approach.  
\begin{theorem}
  Let $S \subseteq M(K)$ be a finite set of places of $K$ that
  contains all Archimedean places.  Let $v_0 \in S$ be a distinguished
  place and choose a continuation of $|\cdot |_{v_0}$ to
  $\overline{\mathbb Q}$, also denoted $|\cdot |_{v_0}$.  Given
  $m\geq 2$, let $L(x_1,\ldots,x_{m})$ be a linear form with algebraic
  coefficients and let $i_0 \in \{1,\ldots,m\}$ be a distinguished
  index such that $x_{i_0}$ has non-zero coefficient in
  $\overline{\mathbb Q}$.  Then for
  any $\varepsilon>0$ the set of solutions
  $\boldsymbol{a}=(a_1,\ldots,a_m) \in (\mathcal{O}_S)^m$ of the
  inequality
  \[ |L(\boldsymbol{a})|_{v_0}\cdot \Bigg( \prod_{\substack{(i,v)\in
        \{1,\ldots,m\}\times S\\(i,v) \neq (i_0,v_0)}}|a_i|_v \Bigg)
    \leq H(\boldsymbol{a})^{-\varepsilon} \] is contained in a finite union
  of proper linear subspaces of $K^m$.
\label{thm:SUBSPACE}
\end{theorem}

We will need the following proposition about roots of
univariate polynomials.
\begin{proposition}{\cite[Proposition 2.3]{LEN97}}
  Let $f \in K[x,x^{-1}]$ be a Laurent polynomial with at most 
  $k+1$ terms.  Assume that $f$ can be written as the sum of two 
  polynomials $g$ and $h$, where every monomial of $g$ has degree at 
  most $d_0$ and every monomial of $h$ has degree at least $d_1$. 
  Let $\beta$ be a root of $f$ that is not a root of unity.  If 
  $d_1-d_0> \frac{\log (k \, H(f))}{\log H(\beta) }$ then $\beta$ is a 
  common root of $g$ and $h$. 
\label{prop:gap}
  \end{proposition}

   We will also need the following lower bound~\cite[Chapter 3]{Wald2000}.
  \begin{proposition}
    Let $f \in \mathbb Z[x]$ have degree $d$ and let $L$ be the sum
    of the absolute value of its coefficients.  If
    $\beta \in \overline{\mathbb Q}$ is not a root of $f$ then
\[ |f(\beta)| > \frac{1}{c^d L^{\mathrm{deg}(\beta)}} \, , \]
where $c$ is a  constant that depends only on the height of $\beta$.
\label{prop:lower}
\end{proposition}

 \section{Echoing Words}
\label{sec:motivate}
In this section we present the main definition of the paper---the
notion of echoing word.  To motivate this we first
present an informal analysis of periodicity properties of the
Fibonacci and Tribonacci words: two well-known examples of Pisot
morphic words.

\subsection{The Fibonacci Word}
\label{sec:fib}
Let $\Sigma=\{0,1\}$ and consider the morphism 
$\varphi :\Sigma^+\rightarrow \Sigma^+$ given by 
$\varphi(0)=01$ and $\varphi(1)=0$.
The \emph{Fibonacci word} $\boldsymbol u_{\mathrm{Fib}} \in \Sigma^\omega$
is the morphic word
\[ \boldsymbol u_{\mathrm{Fib}}  := \lim_{n\rightarrow \infty}\varphi^n(0) =
  01001010010010100 \ldots \,  . \]

The Fibonacci word is not periodic and hence $\boldsymbol u_{\mathrm{Fib}}$ is not equal
to any of its tails $\sigma^n(\boldsymbol u_{\mathrm{Fib}})$ for $n >0$.  However, if the
shift $n$ is well chosen then the mismatches between $\boldsymbol u_{\mathrm{Fib}}$
and $\sigma^n(\boldsymbol u_{\mathrm{Fib}})$ are sparse.  This intuition
will be formalised in the
definition of echoing word.  It turns out that a
particularly good choice of shifts are the elements of the sequence
$\langle 1,2,3,5,8\ldots\rangle$ of Fibonacci numbers: for example,
vertically aligning $\boldsymbol u_{\mathrm{Fib}}$ and $\sigma^5(\boldsymbol u_{\mathrm{Fib}})$ and underlining
mismatches, we have:
  \begin{align*}
 \boldsymbol u_{\mathrm{Fib}} := \,  & 010010\underline{10}01001010010\underline{10}010010\underline{10}0100101001 
      \ldots \\
 \sigma^5(\boldsymbol u_{\mathrm{Fib}}) := \,   & 010010\underline{01}01001010010\underline{01}010010\underline{01}0100101001 \ldots 
  \end{align*}
  Observe that each mismatch involves a factor 10 of $\boldsymbol u_{\mathrm{Fib}}$
  for which the corresponding factor in 
  $\sigma^5(\boldsymbol u_{\mathrm{Fib}})$ is the reverse, 01.  In fact, we see
  the same phenomenon for all shifts of $\boldsymbol u_{\mathrm{Fib}}$ by an element of
  the Fibonacci sequence.  Furthermore, it turns out that for each
  successive shift, the distance between the  mismatched factors
  increases.  This is formalised below as the \emph{expanding gaps
    property}. 
  \subsection{The Tribonacci Word}
\label{sec:trib}
  The following example illustrates some more complicated phenomena
  related to matching a substitutive sequence against shifts of itself.
Let $\Sigma=\{0,1,2\}$ and consider the morphism 
$\varphi :\Sigma^+\rightarrow \Sigma^+$ given by 
$\varphi(0)=01$, $\varphi(1)=02$, and $\varphi(2)=0$. 
The \emph{Tribonacci word} $\boldsymbol u_{\mathrm{Tri}} \in \Sigma^\omega$
is the morphic word 
\[ \boldsymbol u_{\mathrm{Tri}} := \lim_{n\rightarrow \infty}\varphi^n(0) = 0102010010201 \ldots \, .\]
Associated with the Tribonacci word we have the sequence
$\langle t_n \rangle_{n=0}^\infty$ of Tribonacci numbers, defined by
the recurrence $t_n = t_{n-1}+t_{n-2}+t_{n-3}$ and initial conditions
$t_0=1,t_1=2,t_2=4$.  It is easy to see that the word $\varphi^n(0)$ has length $t_n$ for all
$n\in\mathbb{N}$.  

In the spirit of our analysis of the Fibonacci word, we match the
Tribonacci word against shifts of itself by elements of the Tribonacci
sequence $\langle 1,2,4,7,13,24,\ldots \rangle$.  For example,
comparing $\boldsymbol u_{\mathrm{Tri}}$ and $\sigma^{13}(\boldsymbol u_{\mathrm{Tri}})$ we have:
\begin{align*}
  \boldsymbol u_{\mathrm{Tri}} := &
01020100102010\underline{102}0\underline{10}0\underline{102}01020100102010\underline{102010}01020100102010\underline{102}
                \ldots \\
  \sigma^{13}(\boldsymbol u_{\mathrm{Tri}}) := &
   01020100102010\underline{201}0\underline{01}0\underline{201}01020100102010\underline{010201}01020100102010\underline{201}
                                  \ldots
\end{align*}                                  
Notice that the mismatches, above,
appear as a fixed set of factors (either 10, 20, or 102) in $\boldsymbol u_{\mathrm{Tri}}$
that get reversed in $ \sigma^{13}(\boldsymbol u_{\mathrm{Tri}})$.  Unlike with the
Fibonacci word, this time the factors may appear close to each other.
Nevertheless, by suitably grouping neighbouring mismatch factors, we recover a form of
the expanding gaps property and we are moreover able to show that the
mismatches between $\boldsymbol u_{\mathrm{Tri}}$ and its shifts are relatively sparse.

\subsection{Definition of Echoing Words}
      \label{sec:echoing}
      Drawing on the respective examples of the Fibonacci and
      Tribonacci words, we give in this section the formal definition
      of echoing word.  To this end, we introduce the following
      notation and terminology.  Given two non-empty intervals
      $I,J \subseteq \mathbb{N}$, write $I<J$ if $a<b$ for all
      $a\in I$ and $b\in J$, and define the distance of $I$ and $J$ to
      be $d(I,J):=\min\{|a-b|:a\in I,b\in J\}$.  Furthermore define
      the density of a non-empty finite set $S\subseteq \mathbb N$ to be
      $\mathrm{den}(S):=\frac{|S|}{\max(S)}$.  For two functions
      $f(n)$ and $g(n)$ we write $f\ll g$ if there exist positive
      constants $n_0,c_1$ such that for all $n\geq n_0$ we have
      $f(n) \leq c_1 g(n)$.  We also write $f\asymp g$ if both $f\ll g$
      and $g\ll f$.
      
      \begin{definition}
        Let $\Sigma:=\{0,\ldots,b-1\}$ be a finite alphabet. An
        infinite word
        $\boldsymbol{u}=u_0u_1u_2\ldots \in \Sigma^{\omega}$ is
        said to be \emph{echoing} if given $\varepsilon>0$ there exist integers
        $0\leq r_n < s_n$
        and non-empty intervals  $\{0\} = I_{n,0} < I_{n,1} < I_{n,2}
        < \cdots$, for every $n\in\mathbb N$, with the following properties:
        \begin{enumerate}
        \item (Covering): 
  $\left\{ i \in \mathbb N : u_{i+s_n} \neq u_{i+r_n} \right\}
  \subseteq  \bigcup_{j=1}^\infty I_{n,j}$; 
\item  (Density): $\mathrm{den}\left(\bigcup_{j=1}^\delta   I_{n,j} \right) \leq 
    \varepsilon$ for all sufficiently large $\delta$ and $n$;
  \item (Expanding Gaps): we have $s_n<s_{n+1}$ for all $n$ and
    as $n$ tends to infinity we have
    $s_n-r_n \gg s_n$, and $d(I_{n,j},I_{n,j+1}) \gg s_n$, where the implied
    constants are independent of $j \in \mathbb N$.
\end{enumerate}
  \label{def:echoing}
\end{definition}
The intuition behind this definition is as follows.  Conditions 1 and
2 concern the alignment of the suffixes $\sigma^{r_n}(\boldsymbol u)$
and $\sigma^{s_n}(\boldsymbol u)$ of $\boldsymbol u$, where
$0\leq r_n< s_n$.  They say that the set of discrepancies between
these two suffixes is contained in a union of intervals of asymptotic
density at most $\varepsilon$.  Condition 3 says that as $n$ tends to
infinity the distance between consecutive intervals grows linearly with
$s_n$.

We will see that the Fibonacci and Tribonacci words are both echoing.
As suggested by the examples above, in the case of the Fibonacci word
a suitable choice of the sequence $\langle s_n\rangle_{n=0}^\infty$
of shifts, is the sequence of Fibonacci numbers, while in the case of
the Tribonacci word it is the sequence of Tribonacci numbers.  In
the case of the Fibonacci word the intervals $I_{n,j}$ are doubletons,
whereas in the case of the Tribonacci word their total length grows
with $n$,  although the asymptotic density of $\bigcup_{j=1}^\infty
I_{n,j}$ tends to zero as $n$ tends to infinity.

\section{Binary Pisot Morphic Words Are Echoing}
\label{sec:binary}
\subsection{Balanced Pairs}
A pair of words $(x,y) \in \Sigma^+\times \Sigma^+$ is said to be
\emph{balanced} if $x$ and $y$ have the same commutative image.  A
balanced pair is said to be \emph{irreducible} if it cannot be
decomposed as a product (in the semigroup $\Sigma^+\times \Sigma^+$)
of two or more balanced pairs.  For $x \in \Sigma$, we call the
balanced pair $(x,x)$ is a \emph{coincidence}.  An irreducible
balanced pair that is not a coincidence is called a \emph{mismatch}

Let $\varphi:\Sigma^+ \rightarrow \Sigma^+$ be a morphism.  Notice that
if $(x,y)$ is a balanced pair then so is
$(\varphi(x),\varphi(y))$.  A finite set
$\Gamma \subseteq \Sigma^+\times \Sigma^+$ of irreducible balanced
pairs is said to be \emph{closed for $\varphi$} if for all
$(x,y) \in \Gamma$ we can write
$(\varphi(x),\varphi(y))$ as a product
of balanced pairs in $\Gamma$.
If $\Gamma$ is closed then $\varphi$ naturally lifts to a morphism
$\widetilde{\varphi}$ on $\Gamma$ such that
$\widetilde{\varphi}((x,y))$ is  the decomposition of
$(\varphi(x),\varphi(y))$ as a product of elements of
$\Gamma$.  We note that $\widetilde{\varphi}$ necessarily maps a
coincidence to a string of coincidences.  We say that
$\widetilde{\varphi}$  satisfies the \emph{coincidence condition} if
for all $(x,y) \in \Gamma$ there exists $k\geq 0$ such that
the string $\widetilde{\varphi}^k((x,y))$ contains a
coincidence.  It is not difficult to see that if $\widetilde{\varphi}$
is satisfies the \emph{coincidence condition} then the asymptotic
upper density of mismatch symbols in
$\widetilde{\varphi}^k((x,y))$ decays to zero at an exponential rate
in $k$.

Let $\boldsymbol u$ be a fixed point of a primitive morphism
$\varphi$.  Given a non-empty prefix $x$ of
$\boldsymbol u$, since $\boldsymbol u$ is uniformly recurrent we can
write
$\boldsymbol u = x  y x \boldsymbol
u'$ for some $y \in \Sigma^*$ and suffix
$\boldsymbol u' \in \Sigma^{\omega}$.  Given as input the prefix $x$,
the \emph{balanced-pair algorithm}~\cite{Livshits1987,Livshits1992}
successively computes all irreducible factors of the balanced pairs
$(\varphi^k(xy),\varphi^k(yx))$ for
$k=1,2,\ldots$.  We say that the algorithm \emph{terminates with
  coincidence} if the resulting set $\Gamma$ of irreducible factors is
finite and the lifting of $\varphi$ to a morphism
$\widetilde{\varphi}$ on $\Gamma$ satisfies the coincidence condition.

Concerning the termination of the balanced-pair algorithm,
we have the following result of Hollander and 
 Solomayak~\cite{Hollander2003TwosymbolPS} (see also~\cite[Theorem 5.3]{Akiyama2015}).
 \begin{theorem}[Hollander and Solomayak]
Given an irreducible Pisot substitution $\varphi$, the substitutive
dynamical system $(X_\varphi,\sigma)$ has pure discrete spectrum
if and only if the balanced-pair algorithm for $\varphi$ terminates
with coincidence.
\label{thm:HS}
\end{theorem}
Based on the above result and work of Barge and
Diamond~\cite{Barge2002}, we have the following result in the case of
binary alphabets:
    \begin{theorem}[Barge and Diamond]
      Let $\varphi$ be an irreducible Pisot morphism on a two-letter
      alphabet with fixed point $\boldsymbol u$.  Then there exists a
      non-empty prefix $x$ of $\boldsymbol u$ such that balanced-pair
      algorithm terminates with coincidence on input $x$.
  \label{thm:BD}
\end{theorem}

\subsection{The Tribonacci Word is Echoing}
\label{sec:echo}
This section presents an extended example of an echoing word.
The construction is only sketched below and will be treated more formally and in more
generality in Lemma~\ref{lem:binary} and Theorem~\ref{thm:kbon}.

Consider the alphabet $\Sigma=\{0,1,2\}$ and the Tribonacci morphism
  $\varphi : \Sigma^+\rightarrow \Sigma^+$, defined by 
  $\varphi(0)=01$, $\varphi(1)=02$, and $\varphi(2)=0$.
We introduce the alphabet
$\Gamma := \{a_0,\ldots,a_{10}\}$ of irreducible balanced pairs, whose
elements  (depicted as tiles) are as follows:
\[     a_0 = \begin{bmatrix} 0&1\\1&0 \end{bmatrix} \;\;
    a_1 = \begin{bmatrix} 1&0\\0&1 \end{bmatrix} \;\; 
    a_2 = \begin{bmatrix} 0&2\\2&0 \end{bmatrix} \;\; 
a_3 = \begin{bmatrix} 2&0\\0&2 \end{bmatrix} \;\; 
a_4 = \begin{bmatrix}
                                         1&0&2\\2&0&1\end{bmatrix} \quad                                     
                                 a_5 = \begin{bmatrix}
                                   2&0&1\\1&0&2\end{bmatrix} \]
\[ a_{6} = \begin{bmatrix} 0&1&0&2 \\ 2&0&1&0 \end{bmatrix}\quad 
a_{7} = \begin{bmatrix} 2&0&1&0 \\ 0&1&0&2 \end{bmatrix}\quad 
a_{8} = \begin{bmatrix} 0\\0 \end{bmatrix} \quad 
  a_{9} = \begin{bmatrix} 1\\1 \end{bmatrix} \quad                                                  
  a_{10} = \begin{bmatrix} 2\\2 \end{bmatrix} \, .
\]
We partition $\Gamma$ into the subset $\Gamma_0:=\{a_0,\ldots,a_7\}$ of
mismatches and the subset $\Gamma_1:=\{a_8,a_9,a_{10}\}$ of
coincidences.

The map $\varphi$ lifts to a morphism
$\widetilde{\varphi} : \Gamma^+\rightarrow \Gamma^+$ that is defined as follows:
\[\begin{array}{r@{}c@{}lr@{}c@{}l}
    \widetilde{\varphi}\left( \begin{bmatrix} 0&1\\1&0 \end{bmatrix} 
                                                   \right)& \, := \, & \begin{bmatrix} 0\\0 \end{bmatrix} 
    \begin{bmatrix} 1&0&2\\2&0&1\end{bmatrix} & 
        \widetilde{\varphi}\left( \begin{bmatrix} 1&0&2\\2&0&1 \end{bmatrix} \right) 
                                                   &:=& \begin{bmatrix} 0\\0 \end{bmatrix} 
    \begin{bmatrix} 2&0&1&0\\0&1&0&2 \end{bmatrix} \\                                                
    \widetilde{\varphi}\left( \begin{bmatrix} 1&0\\0&1 \end{bmatrix} 
                                                   \right) &\, :=\, & \begin{bmatrix} 0\\0 \end{bmatrix} 
    \begin{bmatrix} 2&0&1\\1&0&2 \end{bmatrix}      &
    \widetilde{\varphi}\left( \begin{bmatrix} 2&0&1\\1&0&2 \end{bmatrix} \right) 
                                                   &:=& \begin{bmatrix} 0\\0 \end{bmatrix} 
    \begin{bmatrix} 0&1&0&2\\2&0&1&0 \end{bmatrix} \\
    \widetilde{\varphi}\left( \begin{bmatrix} 0&2\\2&0 \end{bmatrix} \right) 
                                                   &:=& \begin{bmatrix} 0\\0 \end{bmatrix} 
    \begin{bmatrix} 1&0\\0&1 \end{bmatrix}      &
            \widetilde{\varphi}\left( \begin{bmatrix} 0&1&0&2\\2&0&1&0 \end{bmatrix} \right) 
                                                   &:=& \begin{bmatrix} 0\\0 \end{bmatrix} 
    \begin{bmatrix} 1&0\\0&1\end{bmatrix} \begin{bmatrix}
      2&0\\0&2\end{bmatrix} \begin{bmatrix} 1&0\\0&1\end{bmatrix}\\
    \widetilde{\varphi}\left( \begin{bmatrix} 2&0\\0&2 \end{bmatrix} \right) 
                                                   &:=& \begin{bmatrix} 0\\0 \end{bmatrix} 
    \begin{bmatrix} 0&1\\1&0 \end{bmatrix}  &
                   \widetilde{\varphi}\left( \begin{bmatrix} 2&0&1&0\\0&1&0&2 \end{bmatrix} \right) 
                                                   &:=& \begin{bmatrix} 0\\0 \end{bmatrix} 
    \begin{bmatrix} 0&1\\1&0\end{bmatrix} \begin{bmatrix}
      0&2\\2&0\end{bmatrix} \begin{bmatrix} 0&1\\1&0\end{bmatrix}                                                
  \end{array}\]
and 
\[\widetilde{\varphi}\left( \begin{bmatrix} 0\\0 \end{bmatrix}\right) :=
  \begin{bmatrix} 0\\0 \end{bmatrix} \begin{bmatrix}
    1\\1 \end{bmatrix} \quad\qquad
  \widetilde{\varphi}\left( \begin{bmatrix} 1\\1 \end{bmatrix}\right)
  :=\begin{bmatrix} 0\\0 \end{bmatrix} \begin{bmatrix}
    2\\2 \end{bmatrix} \quad\qquad
  \widetilde{\varphi}\left( \begin{bmatrix} 2\\2 \end{bmatrix}\right)
  := \begin{bmatrix} 0\\0 \end{bmatrix} \, . \] By inspection,
$\widetilde{\varphi}$ satisfies the coincidence condition.

      Let the projection morphisms $\pi_1,\pi_2 :\Gamma^+ \rightarrow \Sigma^+$
      be defined by $\pi_1(x,y):=x$ and $\pi_2(x,y):=y$ for all
      $(x,y)\in \Gamma$.
      Denote by $\boldsymbol u \in \Sigma^{\omega}$ the Tribonacci
      word and by 
      \[ \boldsymbol w = \begin{bmatrix} 0&1\\1 & 0 \end{bmatrix}
        \begin{bmatrix} 0&2\\2 & 0 \end{bmatrix}
        \begin{bmatrix} 0&1\\1 & 0 \end{bmatrix}
        \begin{bmatrix} 0 \\ 0 \end{bmatrix}
        \begin{bmatrix} 0&1\\1 & 0 \end{bmatrix}
        \begin{bmatrix} 0&2\\2 & 0 \end{bmatrix}
                \begin{bmatrix} 0&1\\1 & 0 \end{bmatrix}
                \begin{bmatrix} 0&1\\1 & 0 \end{bmatrix}
                \cdots \in \Gamma^{\omega}
                \]
the image of $\boldsymbol u$
      under the coding $\iota:\Sigma^+\rightarrow \Gamma^+$ given
      by
      \[ \iota(0) :=         \begin{bmatrix} 0&1\\1 &
          0 \end{bmatrix}\qquad
        \iota(1) :=         \begin{bmatrix} 0&2\\2 & 0 \end{bmatrix}
        \qquad
        \iota(2):=         \begin{bmatrix} 0 \\ 0 \end{bmatrix}\, .
        \]
        Then
      $\pi_1(\boldsymbol w) = \boldsymbol u$ and
     $\pi_2(\boldsymbol w)=\sigma(\boldsymbol u)$.
Furthermore, 
      for $n_0 \in \mathbb N$ we can write
      \begin{gather}
\widetilde{\varphi}^{n_0}(\boldsymbol w) =   w_0 a_{i_1}
 {w}_{1} a_{i_2}   w_{2} a_{i_3} \cdots 
\label{eq:partition}
\end{gather}
where $ {w}_0, {w}_1,\ldots \in \Gamma^*$ are sequences of coincidence
symbols and $a_{i_1},a_{i_2,},\ldots \in \Gamma$ are mismatch symbols.
Since $\widetilde{\varphi}$ satisfies the coincidence condition, by
making $n_0\in\mathbb{N}$ sufficiently large the upper asymptotic density of
mismatch symbols in $\widetilde{\varphi}^{n_0}(\boldsymbol{w})$ can be
made arbitrarily small.

The data to show the echoing property are as
follows.
For all $n\in \mathbb{N}$,
consider the image of~\eqref{eq:partition} under the $n$-fold
application of $\widetilde{\varphi}$:
    \begin{gather} \widetilde{\varphi}^{n+n_0}(\boldsymbol{w}) = 
  \widetilde{\varphi}^{n}(w_0) \underbrace{\widetilde{\varphi}^{n}(a_{i_1})}_{I_{n,1}}
  \widetilde{\varphi}^{n}(w_1)  \underbrace{\widetilde{\varphi}^{n}(a_{i_2})}_{I_{n,2}}
  \widetilde{\varphi}^{n}(w_2)
  \underbrace{\widetilde{\varphi}^{n}(a_{i_3})}_{I_{n,2}} \cdots 
\label{eq:PART22} \, .
\end{gather}
We define $r_n:=0$ and $s_n:=|\varphi^{n+n_0}(0)|$.  Then
$\pi_1(\widetilde{\varphi}^{n+n_0}(\boldsymbol w)) =\boldsymbol u$ and
$\pi_2(\widetilde{\varphi}^{n+n_0}(\boldsymbol w))
=\sigma^{s_n}(\boldsymbol u)$.  For all $j\in \{1,2,\ldots\}$ we
define $I_{n,j} \subseteq \mathbb{N}$ to be the smallest interval of
positions in
$\boldsymbol u=\pi_1(\widetilde{\varphi}^{n+n_0}(\boldsymbol w))$ that
contains all mismatch symbols occurring in the factor
$\pi_1(\widetilde{\varphi}^{n}(a_{i_j}))$ (see~\eqref{eq:PART22} for
an illustration).  By construction, this choice satisfies the Covering
Property.  Intuitively, the Density Property follows from the fact that,
by the primitivity of $\varphi$,
for all $n\in \mathbb N$ the upper asymptotic density of
$\bigcup_{j=1}^\infty I_{n,j}$ is bounded above by a constant multiple
of the upper asymptotic density of mismatch symbols in
$\widetilde{\varphi}^{n_0}(\boldsymbol{w})$.
The primitivity of $\varphi$ can also be used to show that
$d(I_{n,j},I_{n,j+1}) \gg s_n$ for all $j\geq 1$.  Note that this bound
holds even when the factor $w_j$ in~\eqref{eq:partition} is empty, since
$\widetilde{\varphi}(a_{i_{j+1}})$ starts with a coincidence symbol.
We thus have the Expanding-Gaps Property.

\subsection{Pisot Morphic Words are Echoing}

The following is the main result of this section.
\begin{theorem}
  Let $\boldsymbol u$ be a fixed point of an irreducible Pisot 
  morphism $\varphi$ over alphabet $\Sigma$.
  Assuming the Pisot conjecture, it holds that
 $\boldsymbol u$ is echoing.   In the case that $\Sigma$ is a binary
 alphabet,  $\boldsymbol u$ is echoing unconditionally. 
\label{thm:binary}
\end{theorem}

Theorem~\ref{thm:binary} follows from Theorems~\ref{thm:HS}
and~\ref{thm:BD} (which respectively give conditions for the
balanced-pair algorithm to terminate with coincidence over binary and
general alphabets) and the following lemma (showing that termination
with coincidence implies the echoing condition).  The proof of
Lemma~\ref{lem:binary} generalises the construction for the Tribonacci
word in Section~\ref{sec:echo}.
  
\begin{lemma}
  Let $\boldsymbol u$ be a fixed point of a primitive morphism $\varphi$
  over alphabet $\Sigma$.  If there is a finite set $\Gamma$ of
  balanced pairs such that $\varphi$ lifts to a morphism
  $\widetilde{\varphi} : \Gamma^+\rightarrow \Gamma^+$ satisfying the
  coincidence condition then
  $\boldsymbol u$ is either eventually periodic or echoing.
\label{lem:binary}
\end{lemma}
  \begin{proof}
    By assumption, there exists a finite alphabet
    $\Gamma \subseteq \Sigma^+\times \Sigma^+$ of irreducible balanced
    pairs and a lifting of $\varphi$ to a morphism
    $\widetilde{\varphi} : \Gamma^+ \rightarrow \Gamma^+$ satisfying
    the coincidence condition.  More specifically we suppose that $\Gamma$ is the output of
    the balanced-pair algorithm on input $x$ a non-empty prefix of
    $\boldsymbol u$.  Since $x$ occurs in $\boldsymbol u$ with bounded
    gaps we may 
    write $\boldsymbol u = x y_1 x y_2  x \cdots$ for some
    $y_1,y_2,\ldots \in \Sigma^*$ and obtain an
    infinite word $\boldsymbol w \in \Gamma^{\omega}$ by
    decomposing the infinite sequence of balanced pairs 
    \begin{gather} \begin{bmatrix}   x   y_1\\   y_1   x \end{bmatrix}
      \begin{bmatrix}   x   y_2 \\   y_2   x \end{bmatrix}
      \begin{bmatrix}   x   y_3 \\   y_3   x \end{bmatrix}
      \cdots
\label{eq:split}
    \end{gather}
into a sequence of irreducible balanced pairs in  $\Gamma$.
Notice that $\pi_1(\boldsymbol w) = \boldsymbol u$ and
    $\pi_2(\boldsymbol w) = \sigma^{\ell}(\boldsymbol u)$, where
    $\ell=|x|$.
    Observe also that $\boldsymbol w$ is uniformly recurrent, being a morphic image of
the word shown in~\eqref{eq:split} (which in turn inherits uniform
recurrence from $\boldsymbol u$).
    Moreover, since $\varphi$ is primitive, there exists $c_0>0$ 
    such that for all
    $\gamma,\gamma' \in \Gamma$ and all $n\in
    \mathbb N$, we have
    \begin{gather}
    |\pi_1(\widetilde{\varphi}^n(\gamma))| \leq c_0\,
    |\pi_1(\widetilde{\varphi}^n(\gamma'))| \, .
\label{eq:growth}
  \end{gather}

      Assume that $\boldsymbol u$ is not ultimately periodic.  Then
      given $n_0 \in \mathbb{N}$ the word
      $\widetilde{\varphi}^{n_0}(\boldsymbol w)$ contains infinitely
      many mismatch symbols.  Moreover, since $\widetilde{\varphi}$
      satisfies the coincidence condition, for $n_0$ sufficiently
      large, $\widetilde{\varphi}^{n_0}(\boldsymbol w)$ also contains
      infinitely many coincidence symbols.  For such $n_0$ we thus
      have either
  \begin{gather}\widetilde{\varphi}^{n_0}(\boldsymbol w)= { z}_0 { v}_0 { z}_1 
    { v_1} { z}_2 \cdots
    \quad\text{or}\quad
\widetilde{\varphi}^{n_0}(\boldsymbol w)= { v}_0 { z}_0 { v}_1 
    { z_1} { v}_2 \cdots
\label{eq:partitionA}
  \end{gather}
  where ${ z}_0,{ z}_1,\ldots \in \Gamma^+$ are non-null sequences of
  mistmatch symbols and ${ v}_0,{ v}_1,\ldots \in \Gamma^+$ are non-null
  sequences of coincidence symbols. 
  Since $\boldsymbol w$ is
  uniformly recurrent there is an upper bound on the length of the
  factors $ z_i$ and $ v_i$ that depends only on
  $n_0$.  Let $\varepsilon>0$ be given as in the definition of an
  echoing sequence.  Since $\widetilde{\varphi}$ satisfies the
  coincidence condition, for sufficiently large $n_0\in\mathbb{N}$ the
  upper asymptotic density of mismatch symbols in
  $\widetilde{\varphi}^{n_0}(\boldsymbol{w})$ is at most
  $\varepsilon/c_0$.

  The data to show the echoing property are as follows
  (see Definition~\ref{def:echoing}).  Let us assume the left-hand case
  in~\eqref{eq:partitionA}; the reasoning for the right-hand case
  follows with minor changes.  For all $n\in \mathbb{N}$, consider the
  string \begin{gather} \widetilde{\varphi}^{n+n_0}(\boldsymbol{w}) =
    \widetilde{\varphi}^{n}( z_0) \widetilde{\varphi}^{n}( v_0)
    \underbrace{ \widetilde{\varphi}^{n}({ z}_1)}_{I_{n,1}}
    \widetilde{\varphi}^{n}( v_1) \underbrace{
      \widetilde{\varphi}^{n}({ z}_2)}_{I_{n,2}}
    \widetilde{\varphi}^{n}( v_2) \cdots
\label{eq:PART2} \, . 
\end{gather}
Then for all $n\in \mathbb N$ we have
$\pi_1(\widetilde{\varphi}^{n+n_0}(\boldsymbol w)) = \boldsymbol u$
and $\pi_2(\widetilde{\varphi}^{n+n_0}(\boldsymbol w)) =
\sigma^{t_n}(\boldsymbol u)$, where $t_n:= |\varphi^{n+n_0}(x)|$.

For all $n\in\mathbb N$, define
$r_n:=|\pi_1(\widetilde{\varphi}^n( z_0))|$ and $s_n:=r_n+t_n$.  For
all $j\in \{1,2,\ldots\}$, we define $I_{n,j} \subseteq \mathbb{N}$ to
be the interval of positions in
$\boldsymbol u =\pi_1(\widetilde{\varphi}^{n+n_0}(\boldsymbol w))$
corresponding to the factor $\pi_1(\widetilde{\varphi}^{n}(z_j))$
(see~\eqref{eq:PART2}).  By construction, this choice satisfies the
Covering Property.  Since the mismatch symbols in
$\widetilde{\varphi}^{n_0}(\boldsymbol w)$ have upper asymptotic
density at most $\varepsilon/c_0$, by~\eqref{eq:growth} we have
$\mathrm{den}\left(\bigcup_{j=1}^\delta I_{n,j}\right)\leq
\varepsilon$ for $\delta$ sufficiently large; hence the Density
Property holds.  Finally, $\pi_1(\widetilde{\varphi}^n(\gamma))$ has
the same asymptotic growth rate for all $\gamma \in \Gamma$ by
primitivity of $\varphi$.  It follows that
$d(I_{n,j},I_{n,j+1}) \gg s_n$ for all $j\geq 1$ and $s_n-r_n\gg s_n$,
where the implied constant in the first bound is indepedent of $j$.
This establishes the Expanding-Gaps Property.
  \end{proof}

\section{Transcendence for Echoing Words}
     \label{sec:transcendence}
We have the following transcendence result for echoing words.
     \begin{theorem}
       Let $\Sigma:= \{0,1,\ldots, b-1\}$ and suppose that 
  the sequence
  $\boldsymbol{u}=\langle u_m\rangle_{m=0}^\infty \in \Sigma^{\omega}$ is echoing.  Then
  for any algebraic number $\beta$ such that $|\beta|>1$,  the sum
  $\alpha:=\sum_{m=0}^\infty \frac{u_m}{\beta^m}$ is either
  transcendental or lies in $\mathbb Q(\beta)$.
\label{thm:main}
\end{theorem}
\begin{proof}
  Suppose that $\alpha$ is algebraic.  We will use the Subspace
  Theorem to show that $\alpha \in \mathbb Q(\beta)$.  Let $S$
  comprise all the Archimedean places of $\mathbb Q(\beta)$ and all
  non-Archimedean places corresponding to prime ideals $\mathfrak{p}$
  of $\mathcal O_{\mathbb Q(\beta)}$ such that
  $\mathrm{ord}_{\mathfrak{p}}(\beta)\neq 0$.  Let $v_0\in S$ be the
  place corresponding to the inclusion of $\mathbb Q(\beta)$ in
  $\mathbb{C}$.  Recall that $|\cdot|_{v_0}=|\cdot|^{e_0}$, where
  $|\cdot|$ denotes the usual absolute value on $\mathbb{C}$ and
  $e_0=\frac{1}{\mathrm{deg}(\beta)}$ if $\beta$ is real and
  $e_0=\frac{2}{\mathrm{deg}(\beta)}$  otherwise.  Let
  $\kappa \geq b$ be an upper bound of $|\beta|_v$ for all $v \in S$.

  Let
  $\varepsilon:=\frac{e_0 \log
    |\beta|}{4|S|\log \kappa}$.  By the assumption
  that $\boldsymbol u$ is echoing, for all $n \in \mathbb N$ there
  exist $0\leq r_n < s_n$ and a sequence of intervals
  $\{0\} = I_{n,0} < I_{n,1} < I_{n,2} < \cdots$ satisfying
  Conditions~1--3 of Definition~\ref{def:echoing} for the above choice
  of $\varepsilon$ .
    
  Define
  $\rho:=\frac{4|S|\log \kappa}{e_0 \log |\beta|}$.
  The Expanding-Gaps Condition gives
  $d(I_{n,j},I_{n,j+1}) \gg s_n$ for $j=1,2,\ldots$  with the implied
  constant being independent of $j$; thus there exists $\delta \in \mathbb N$ such that
  for all $n$ the right endpoint $t_n$ of $I_{n,\delta}$ satisfies
  $t_n \geq \rho s_n$.
 For $n\in \mathbb{N}$, define 
    $\boldsymbol{a}_n=(a_{n,1},\ldots,a_{n,\delta+3}) \in (\mathcal{O}_S)^{\delta+3}$ by 
    \begin{equation}
      \begin{aligned}
   & a_{n,1} := \beta^{s_n},\qquad 
       a_{n,2} := \beta^{r_n},\qquad 
       a_{n,3} := \sum_{i=0}^{s_n} {u_i}{\beta^{s_n-i}}  -
       \sum_{i=0}^{r_n} u_i \beta^{r_n-i} \, , \\
   &  a_{n,j+3} := \sum_{i \in I_{n,j}}
    {(u_{i+s_n}-u_{i+r_n})}{\beta^{-i}} \quad (j=1,\ldots,\delta) \, .
  \end{aligned}
\label{eq:TUPLES}
\end{equation}
By passing to a
  subsequence of $n\in \mathbb N$, we may
  assume that there exists $J\subseteq \{1,\ldots,\delta\}$ such that
  for all $j \in \{1,\ldots,\delta\}$ and $n \in \mathbb N$ we have
  $a_{n,j+3}\neq 0$ if and only if $j \in J$.

  Consider the linear form
\begin{gather} L(x_1,\ldots,x_{\delta+3}):= \alpha x_1 - \alpha  x_2 - x_3 -
  \sum_{j \in J} x_{j+3} \,  .
\label{eq:FORM}
\end{gather}
Then we have
\begin{eqnarray}
L(\boldsymbol a_n)   &  =  &
                                                       \left|\alpha\beta^{s_n}-\alpha\beta^{r_n}
                             - \sum_{i=0}^{s_n} u_i\beta^{s_n-i} +
                             \sum_{i=0}^{r_n}u_i\beta^{r_n-i} -  \sum_{j
                             \in J}
                             \sum_{i\in I_{n,j}}
                             (u_{i+s_n}- u_{i+r_n})\beta^i\right| \notag \\
&= &                    \left|\sum_{i=t_n+1}^\infty
                                                              (u_{i+s_n}-u_{i+r_n})\beta^{-i}\right|
  \notag \\
  & \ll  &  |\beta|^{-t_n} \,  .
\label{eq:STRICT}
\end{eqnarray}

To set up the application of the Subspace Theorem, we estimate the
absolute values of the components of $\boldsymbol a_n$.
By the product formula we have $\prod_{v\in S} |a_{n,1}|_v =
\prod_{v\in S} |a_{n,2}|_v = 1$.  Next,
by the facts that $t_n\geq \rho s_n$ and
$\rho:=\frac{4|S|\log \kappa}{e_0 \log |\beta|}$ we have
\begin{gather}
\prod_{v\in S}   |a_{n,3}|_v  \leq \prod_{v\in S} \sum_{i=0}^{s_n}
\kappa^{i+1} \leq \prod_{v\in S} \kappa^{s_n +2} =
  \kappa^{(s_n+2)|S|} \ll |\beta|^{e_0 t_n/4} \, .
\label{eq:BOUND1}
\end{gather}
Next, by the Density Condition, for $\delta$ sufficiently large we have
$\sum_{j=1}^{\delta} |I_{n,j}| \leq \varepsilon t_n$ for all
$n\in\mathbb N$.  By the
product formula and the fact that
$\varepsilon:=\frac{e_0 \log
    |\beta|}{4|S|\log \kappa}$  we have
\begin{gather}
\prod_{j \in J} \prod_{v \in S}
  |a_{n,j+3}|_v \leq \prod_{j \in J} \prod_{v\in S} \sum_{i=0}^{|I_{n,j}|}
  \kappa^{i+1}  \leq \prod_{j \in J}   \kappa^{|S|(2+|I_{n,j}|)} \leq
  \kappa^{2|S|\delta+|S|\varepsilon t_n} \ll
  |\beta|^{e_0t_n/4} \,  .
\label{eq:BOUND2}
\end{gather}
Combining~\eqref{eq:BOUND1} and~\eqref{eq:BOUND2} and the
bound $|L(\boldsymbol{a}_n)|_{v_0} \ll |\beta|^{-e_0t_n}$ from~\eqref{eq:STRICT}, we have
\begin{eqnarray}
  |L(\boldsymbol{a}_n)|_{v_0} \cdot
    \prod_{\substack{(i,v)\in (\{1,2,3\} \cup J) \times S\\(i,v) \neq (1,v_0)}}|a_{n,i}|_v &\ll &
     |\beta|^{e_0 t_n/2 } \cdot 
        |\beta|^{-e_0 t_n} = |\beta|^{-e_0t_n/2}\, ,
\label{eq:INEQ3}
\end{eqnarray}
  
  For all but finitely many $n$, the right-hand side
  of~\eqref{eq:INEQ3} is less than
  $|\beta|^{-e_0t_n/3}$.  On the other hand, there
  exists a constant $c_1>0$ such that the height of $\boldsymbol a_n$
  satisfies $H(\boldsymbol{a}_n) \leq |\beta|^{c_1 t_n}$ for all $n$.
  Thus the right-hand side
  of~\eqref{eq:INEQ3} is at most
  $H(\boldsymbol{a}_n)^{-e_0/3c_1}$ for infinitely many $n$.  We
  can hence apply the Subspace Theorem (Theorem~\ref{thm:SUBSPACE}) to
  obtain a non-zero linear form
  \[ F(x_1,\ldots,x_{3+\delta}) = \alpha_1x_1+
    \alpha_2x_2+\alpha_3x_3+\sum_{j \in J} \alpha_{3+j}x_{3+j}  \] with coefficients in $\mathbb{Q}(\beta)$ such that
  $F(\boldsymbol{a}_n)=0$ for infinitely many $n\in\mathbb{N}$.  In
  other words,
  \begin{gather}
    \alpha_1 \beta^{s_n} + \alpha_2 \beta^{r_n} + \alpha_3
    \left(\sum_{i=0}^{s_n} u_i \beta^{s_n-i} - \sum_{i=0}^{r_n} u_i \beta^{r_n-i}\right)
   + \sum_{j\in J} \alpha_{3+j} a_{3+j} = 0
\label{eq:SIMPLE}
  \end{gather}
  for infinitely many $n$.

  We claim that for $n$ sufficiently large~\eqref{eq:SIMPLE} entails
  that $\alpha_3\neq 0$ and $\alpha_{3+j} =0$ for all $j \in J$.
  Indeed,~\eqref{eq:SIMPLE} entails that $P_n(\beta)=0$ for infinitely
  $n$ for the polynomial
  $P_n(x) :=P_{n,1}(x)+P_{n,2}(x)+ \sum_{j \in J} P_{n,3+j}(x)$, where
\begin{gather*}
  P_{n,1}:=\alpha_1 x^{s_n},\qquad P_{n,2}:=\alpha_2 x^{r_n},\qquad
  P_{n,3}=\alpha_3 \left(\sum_{i=0}^{s_n} u_i x^{s_n-i} -
    \sum_{i=0}^{r_n} u_i x^{r_n-i}\right)\\
  P_{n,3+j} :=\alpha_{3+j}\sum_{i \in I_{n,j}} (u_{s_n+i}-u_{r_n+i}) 
  x^{-i} \quad(j \in J)
\end{gather*}
The height of $P_n$ is bounded independently of $n$ and its degree is
bounded by a constant multiple of $s_n$.  Applying
Proposition~\ref{prop:gap}, we deduce that for $n$ sufficiently large,
if $P_n(\beta)=0$ then each constituent $P_{n,3+j}(\beta)$ must
individually vanish (since the gap between the maximum and minimum
degrees of any two distinct polynomials $P_{n,3+j}(x)$ is $\gg s_n$ by
the Expanding Gaps condition).  But this entails that $\alpha_{3+j}=0$
for all $j \in J$.  Assume further, for a contradiction, that
$\alpha_3=0$.  Applying Proposition~\ref{prop:gap} again, we get
$\alpha_1=\alpha_2=0$, contradicting the non-zeroness of the form $F$.
This completes the proof of the claim.
  
    Having established the claim, we divide~\eqref{eq:SIMPLE} by
    $\beta^{s_n}$.  Then, letting $n$ tend to infinity, we have
    $\alpha_1\alpha_3^{-1}=\sum_{i=0}^\infty u_i\beta^{-i}$.  Hence
    $\sum_{i=0}^\infty u_i\beta^{-i}$ lies in $\mathbb Q(\beta)$.
\end{proof}

\section{Strongly Echoing Sequences}
\subsection{The Non-Vanishing Condition}

The following definition presents a strengthening of the echoing
condition that allows proving transcendence outright (removing the
possibility that $\llbracket\boldsymbol{u}\rrbracket_{\beta}$ be algebraic).  The
change to Definition~\ref{def:echoing} involves strengthening the
Expanding-Gaps Condition (to specify both lower and upper bounds on
the gaps between intervals) and adding a fourth condition, called
\emph{Non-Vanishing}.  The latter guarantees the existence of at least
two intervals of mismatches that are non-zero when evaluated in base~$\beta$.

\begin{definition}
        Let $\Sigma:=\{0,\ldots,b-1\}$ be a finite alphabet. An
        infinite word
        $\boldsymbol{u}=u_0u_1u_2\ldots \in \Sigma^{\omega}$ is
        said to be \emph{strongly echoing} if given $\varepsilon>0$
        for all $n\in\mathbb N$ there
        exist integers $0\leq r_n< s_n$
        and a sequence of non-empty intervals         $\{0\} = I_{n,0} < I_{n,1} < I_{n,2} < \cdots$ with the following 
        properties:
        \begin{enumerate}
        \item (Covering): 
  $\left\{ i \in \mathbb N : u_{i+s_n} \neq u_{i+r_n} \right\}
  \subseteq  \bigcup_{j=1}^\infty I_{n,j}$; 
\item   (Density): $\mathrm{den}\left(\bigcup_{j=1}^\delta   I_{n,j} \right) \leq 
    \varepsilon$ for all sufficiently large $\delta$ and $n$;
 \item (Expanding Gaps): we have $s_n<s_{n+1}$ for all $n$;
    furthermore, as $n$ tends to infinity we have
    $s_n-r_n \gg s_n$, and $d(I_{n,j},I_{n,j+1}) \asymp s_n$, where the implied
    constant is independent of $j \in \{0,1,2,\ldots\}$.
    \item (Non-Vanishing):  for all $\beta\in\overline{\mathbb{Q}}$ such that $|\beta|>1$
  there exist $1\leq j_0<j_1$  
  such that for all  $n\in \mathbb{N}$ and $j\in \{j_0,j_1\}$ we have
  $\sum_{i \in I_{n,j}} (u_{i+s_n}-u_{i+r_n})\beta^{-i}\neq 0$.
\end{enumerate}
  \label{def:echoing2}
\end{definition}

In Section~\ref{sec:NV} we show that the substitutive sequence defined
by the $k$-bonacci morphism is strongly echoing for all $k\geq 2$.
The case $k=2$ (the Fibonacci word) is straightforward.  In following
section we treat the case $k=3$ (the Tribonacci word), continuing the
extended example from Section~\ref{sec:echo}.

\subsection{The Tribonacci Word is Strongly Echoing. }
\label{sec:echo2}
  For $\Sigma:=\{0,1,2\}$, let $\varphi : \Sigma^+\rightarrow \Sigma^+$ be the morphism that
  defines the Tribonacci word $\boldsymbol u \in \Sigma^\omega$ and
  let $\widetilde{\varphi} : \Gamma^+\rightarrow \Gamma^+$ be its
  lifting to the closed set of balanced pairs, as defined in
  Section~\ref{sec:echo}, with
  $\pi_1,\pi_2:\Gamma^+\rightarrow \Sigma^+$ the associated projection
  morphisms.  We recall the construction showing that $\boldsymbol u$
  is echoing, and argue that the non-vanishing condition is also
  satisfied.  In particular, we refer to the shifts 
  $r_n:=0$ and $s_n:=|\varphi^{n+n_0}(0)|$, and intervals $I_{n,j}$ that witness
  that $\boldsymbol u$ is echoing.

  For all $n\in \mathbb N$ and $j\in \{1,2,\ldots\}$, define the
  polynomial $P_{n,j}(x) \in \mathbb{Z}[x]$ by
\begin{gather} P_{n,j}(x) := \sum_{i \in I_{n,j}} (u_{i+s_n}-u_{i+r_n})x^{i} \,
.
\label{eq:Poly}
\end{gather}
Inspecting Condition~4 of Definition~\ref{def:echoing2},  to show that
$\boldsymbol u$ is 
  non-vanishing we must prove that for all~$\beta$ such that 
  $|\beta|>1$  there
  exist  indices $j_0<j_1$ such that $P_{n,j_0}(\beta^{-1})\neq 0$
  and $P_{n,j_1}(\beta^{-1})\neq 0$ for infinitely many $n$.  To this
  end,  we give a recursive
  characterisation of the polynomials $P_{n,j}$.

  Recall the mismatch symbols $a_0,\ldots,a_7$ in $\Gamma$.
  For all $i \in \{0,\ldots,7\}$ and $n\in\mathbb{N}$, define the
  polynomial $Q_{n,i}(x) \in \mathbb Z[x]$ by the equation 
\begin{gather}
  Q_{n,i}(x):=\sum_{k=0}^{\ell-1} (\pi_1(\widetilde{\varphi}^n(a_i))_k -
  \pi_2(\widetilde{\varphi}^n(a_i))_k)x^{k} \,  ,
\label{eq:MATCH}
\end{gather}
where $\ell=|\pi_1(\widetilde{\varphi}(a_i))|$.
For example, we have $Q_{0,6}(x) = -2+x-x^2+2x^3$ since $a_6
= \begin{bmatrix} 0&1&0&2\\ 2&0&1&0
\end{bmatrix}$.  Referring to the list of symbols
$a_{i_1},a_{i_2},\ldots$ and the corresponding 
intervals $I_{n,j}$ in~\eqref{eq:PART22}, we observe that for all
$n\in\mathbb N$ and $j \in \{1,2,\ldots\}$, polynomial $P_{n,j}$ is the
product $Q_{n,i_j}$ and a monomial.  Thus our task is to find
$j_0<j_1$ such that $Q_{n,i_{j_0}}(\beta^{-1})\neq 0$ and
$Q_{n,i_{j_1}}(\beta^{-1})\neq 0$ for infinitely many $n$.
  
We claim that for all $n$ there exists $i\in \{0,\ldots,7\}$ such that
$Q_{n,i}(\beta^{-1}) \neq 0$.  From the claim it follows that there
exists $i^* \in \{0,\ldots,7\}$ such that
$Q_{n,i^*}(\beta^{-1}) \neq 0$ for infinitely many $n \in \mathbb N$.
Since $\widetilde{\varphi}$ acts primitively on the set of mismatch
symbols, for $n_0$ sufficiently large, ${i^*}$ occurs infinitely often
in the list $i_1,i_2,\ldots$ in~\eqref{eq:partition}.  In particular,
there exist $j_0<j_2$ such that $Q_{n,i_{j_0}}(\beta^{-1})\neq 0$ and
$Q_{n,i_{j_1}}(\beta^{-1})\neq 0$ for infinitely many $n$.  This is
what we wanted to prove and it remains to justify the claim.

For all $n\in \mathbb N$ we have the vector 
$\boldsymbol{Q}_n:=(Q_{n,0},\ldots,Q_{n,7}) \in 
\mathbb{Z}[x]^{8}$ of polynomials.  We want to show that
$\boldsymbol{Q}_n(\beta^{-1})\neq 0$ for all $n$.
The action of $\widetilde{\varphi}$ gives rise to a
recurrence for $\boldsymbol Q_n$.  For example, we have
$\widetilde{\varphi}(a_6) = a_8a_1a_3a_1$ and hence
\[ \widetilde{\varphi}^{n+1}(a_6) = \widetilde{\varphi} ^n(a_8)
\widetilde{\varphi} ^n(a_1) \widetilde{\varphi} ^n(a_3)
\widetilde{\varphi} ^n(a_1) \, . \]
We deduce that 
\[Q_{n+1,6}(x)= x^{\ell_{n,8}}Q_{n,1}(x)+ x^{\ell_{n,8}+\ell_{n,1}} Q_{n,3}(x)+
  x^{\ell_{n,8}+\ell_{n,1}+\ell_{n,2}}Q_{n,1}(x)\,  ,\]
where $\ell_{n,i} = |\pi_1(\widetilde{\varphi}^n(a_i))|$ for all $i\in\{0,\ldots,8\}$.
Calculating the corresponding recurrences for each polynomial $Q_{n+1,i}(x)$ 
yields the vector recurrence $\boldsymbol{Q}_{n+1} = M_n \, \boldsymbol{Q}_n$, where
\begin{gather} M_n := x^{\ell_{n,8}} \cdot
  \begin{pmatrix}
 0 &   0        &  0  &     0  & 1  &  0  &  0     & 0      \\
 0 &   0        &  0  &     0  &  0 & 1 &     0 &    0 \\
   0 &   1        &     0 &        0 &    0 &   0 &      0 &  0     \\
 1 &            0 &     0 &        0 &    0 &   0 &      0 &    0   \\
   0 &            0 &     0 &        0 &    0 &   0 &      0 &    1 \\
   0 &            0 &     0 &        0 &    0 &   0 &   1  &   0    \\
  0  &1+x^{\ell_{n,1}+\ell_{n,3}} &     0 & x^{\ell_{n,1}}       &    0 &   0 &      0 &     0  \\
 1+x^{\ell_{n,0}+\ell_{n,2}}   &  0   & x^{\ell_{n,0}}       &      0 &   0& 0 &   0  &      0
\end{pmatrix}
\label{eq:matrix}
\end{gather}

But $Q_{0,0}=(x-1)$ and $\mathrm{det}(M_n)=x^{\ell_{n,0}+\ell_{n,1}+8\ell_{n,8}}$.
Hence for all $\beta\neq 1$ we have
$\boldsymbol{Q}_0(\beta^{-1}) \neq \boldsymbol{0}$ 
and, by induction on $n$,
$\boldsymbol{Q}_n(\beta^{-1}) \neq \boldsymbol{0}$ for all $n\in
\mathbb N$.  The claim is proven and this completes
the argument that $\boldsymbol u$ is strongly echoing.

\subsection{Transcendence Result}
For a strongly echoing sequence $\boldsymbol u$ and algebraic
$\beta$, we can show that $\llbracket\boldsymbol{u}\rrbracket_{\beta}$ is outright
transcendental.  Whereas in the proof of Theorem~\ref{thm:main} we
assume that $\llbracket\boldsymbol{u}\rrbracket_{\beta}$ is algebraic and use the
Subspace Theorem to show that it lies in $\mathbb Q(\beta)$, below we
continue the argument by applying the Subspace Theorem a second time
(employing the Non-Vanishing Condition) to derive a contradiction.

\begin{theorem}
       Let $\Sigma:= \{0,1,\ldots, b-1\}$ and suppose that 
  the sequence
  $\boldsymbol{u}=\langle u_m\rangle_{m=0}^\infty \in \Sigma^{\omega}$ strongly echoing.  Then
  for any algebraic number $\beta$ such that $|\beta|>1$, the sum
  $\alpha:=\sum_{m=0}^\infty \frac{u_m}{\beta^m}$ is transcendental.
\label{thm:main2}
\end{theorem}
\begin{proof}
  We suppose that $\alpha$ is algebraic and obtain a contradiction.
  We refer to the proof of Theorem~\ref{thm:main}, whose notation we
  use throughout this proof.  In particular, we work with the linear form $L$
  in~\eqref{eq:FORM} and tuples
  $\boldsymbol a_n \in (\mathcal O)_S^{\delta+3}$
  in~\eqref{eq:TUPLES}, which satisfy the inequality
  $|L(\boldsymbol a_n)| \ll |\beta|^{-t_n}$ shown in~\eqref{eq:STRICT}.  The only
  deviation from the set up of Theorem~\ref{thm:main} 
  concerns that choice of $\delta \in \mathbb N$, as we now
  explain.

  For the index $j_0$ from the Non-Vanishing Condition, since
  $d(I_{n,j},I_{n,j+1}) \ll s_n$ for all $j<j_0$ by the Expanding-Gaps
  Condition, we can write $a_{n,3+j_0}$ as a polynomial in $\beta$ of
  height at most $b$ and degree bounded by a constant multiple of
  $s_n$.  Applying Proposition~\ref{prop:lower}, there exists a
  constant $c_2>0$ such that $|a_{n,3+j_0}| \gg |\beta|^{-c_2s_n}$.
  By the Expanding Gaps Condition we can choose $\delta$ sufficiently
  large that $\delta \geq j_0,j_1$ and for all $n$ the right endpoint
  $t_n$ of interval $I_{n,\delta}$ satisfies $t_n \geq (\rho+6c_2)s_n$
  for all $n$.
  
  From the proof of Theorem~\ref{thm:main} there is a linear form
  \[ F(x_1,\ldots,x_{3+\delta}) = \alpha_1x_1+
    \alpha_2x_2+\alpha_3x_3+\sum_{j \in J} \alpha_{3+j}x_{3+j} \] with
  algebraic coefficients such that $F(\boldsymbol{a}_n)=0$ for
  infinitely many $n\in\mathbb{N}$, $\alpha_3\neq 0$, and
  $\alpha_{3+j}=0$ for all $j \in J$.  Subtracting a suitable multiple
  of $F$ from $L$ we obtain a new linear form $L'$
  whose support contains $x_{3+j}$ for all $j \in J$ but does not
  contain $x_3$ and that satisfies
  $|L'( \boldsymbol a_n)| \ll |\beta|^{-t_n}$ for infinitely many
  $n$.  Note that the support of $L'$ contains at least two
  variables, namely $x_{3+j_0}$ and $x_{3+j_1}$ for $j_0$ and $j_1$ as
  in the Non-Vanishing Condition.

  Our aim is to apply the Subspace Theorem to $L'$.  To this end,
  from~\eqref{eq:BOUND1} and~\eqref{eq:BOUND2} we have
\begin{gather}
    \prod_{(i,v)\in (\{1,2,3\} \cup J) \times S} |a_{n,i}|_v  \, \ll \,
 |\beta|^{e_0t_n/2} \,  .
\label{eq:BOUND3A}
\end{gather}
Recalling that $|a_{n,3+j_0}| \gg |\beta|^{-c_2s_n}$ and  $|L'( \boldsymbol a_n)| \ll |\beta|^{-t_n}$,
since $t_n \geq 6c_2s_n$ we have
\begin{gather}
|L'(\boldsymbol a_n)|_{v_0} \, |a_{n,3+j_0}|_{v_0}^{-1} \ll |\beta|^{-e_0t_n}
|\beta|^{e_0c_2s_n} \leq |\beta|^{-5e_0t_n/6} \, .
\label{eq:BOUND4A}
\end{gather}
Multiplying~\eqref{eq:BOUND3A} and~\eqref{eq:BOUND4A} we get 
\begin{eqnarray}
 |L'(\boldsymbol{a}_n)|_{v_0} \cdot 
    \prod_{\substack{(i,v)\in (\{1,2,3\} \cup J) \times S\\(i,v) \neq 
        (3+j_0,v_0)}}|a_{n,i}|_v  \ll 
    |\beta|^{-e_0t_n/3} \, .
\label{eq:INEQ4A}
\end{eqnarray}
For all but finitely many $n$, the right-hand side
of~\eqref{eq:INEQ4A} is less than
$|\beta|^{-e_0 t_n/4 }$.  On the other hand, there
exists a constant $c_3>0$ such that the height of $\boldsymbol a_n$
satisfies the bound $H(\boldsymbol{a}_n) \leq |\beta|^{c_3 t_n}$ for
all $n$.  Thus the right-hand
side of~\eqref{eq:INEQ4A} is at most
$H(\boldsymbol{a}_n)^{-e_0/4c_3}$ for infinitely many $n$.  Thus
we may apply the Subspace Theorem to obtain a non-zero linear form
\[ G(x_1,\ldots,x_{3+\delta})
  :=\beta_1x_1+\beta_2x_2+\sum_{j \in J} \beta_{3+j} x_{3+j}\] with
algebraic coefficients whose support does not include $x_3$ such that
$G(\boldsymbol a_n)=0$ for infinitely many $n$.
As shown in the proof of the claim at the end of 
Theorem~\ref{thm:main}, the conditions that $G(\boldsymbol a_n)=0$ for
infinitely many $n$ and that $x_3$ not appear in the support of $G$
entail that $G$ is identically zero, which yields the desired contradiction.
\end{proof}

\section{The $k$-Bonacci Word is Strongly Echoing}
\label{sec:kbon}
\label{sec:NV}

Let $k\geq 2$ and consider the alphabet $\Sigma := \{0,\ldots,k-1\}$.
The $k$-bonacci morphism $\varphi:\Sigma^+ \rightarrow \Sigma^+$ is
defined by
$\varphi(0):=01, \varphi(1):=02, \ldots,\varphi (k-2) = 0\,(k-1)$, and
$\varphi(k-1)=0$.  Let $\boldsymbol u := \lim_{n\rightarrow\infty} \varphi^n(0)$ be the
$k$-bonacci word.  The cases $k=2$ and $k=3$ respectively yield the
Fibonacci and Tribonacci words as considered in Sections~\ref{sec:fib}
and~\ref{sec:trib}.  The main result
of this section shows that $\boldsymbol u$ is strongly echoing.

\begin{theorem}
For $k\geq 2$ the $k$-bonacci word is strongly echoing. 
\label{thm:kbon}
\end{theorem}

The proof of Theorem~\ref{thm:kbon} will be given below.  To prepare
the ground we first describe the output of
the balanced-pair algorithm for the $k$-bonacci word and the
consequent lifting of $\varphi$ to a morphism $\widetilde{\varphi}$ on
balanced pairs.  As with the example of the Tribonacci word, the
key to establishing the Non-Vanishing property is to show that a
certain matrix of polynomials related to $\widetilde{\varphi}$
(generalising~\eqref{eq:matrix}) is
non-singular when evaluated at any number
$\beta\in\mathbb C\setminus \{0,1\}$.

\subsection{Balanced Pairs for the $k$-Bonacci Morphism}
Write $\Sigma_{\bot}:=\Sigma\cup\{-1\}$.
For $a \in \Sigma_\bot$ the \emph{iterated palindromic closure}
$\mathrm{pal}(a) \in \Sigma^*$ is inductively defined by
$\mathrm{pal}(-1):=\varepsilon$ and
$\mathrm{pal}(a) := \mathrm{pal}(a-1) \,a\, \mathrm{pal}(a-1)$ for
$a > -1$.\footnote{Note that 
  $\mathrm{pal}(a)$ denotes the iterated palindromic closure of the
  string $01\cdots a$, as introduced by de Luca~\cite{Luca97a}.}
In manipulating balanced pairs we we write expressions in the free
group generated by $\Sigma$ that simplify to elements 
of the free monoid generated by $\Sigma$.  For example,
for $x,y \in \Sigma^*$ we write $y^{-1}x$ when $y$ is a prefix of
$x$.
\begin{enumerate}
\item For $a \in \Sigma$  we write 
$[a]$ for the balanced pair
$\begin{bmatrix} a\\a \end{bmatrix}$.
\item  For 
  $a>b>c$ in $\Sigma_\bot$, 
$[a,b,c]$ denotes the balanced pair $\begin{bmatrix} a\, \mathrm{pal}(b) \mathrm{pal}(c)^{-1} \\
  \mathrm{pal}(c)^{-1}\mathrm{pal}(b)\,  a  \end{bmatrix}$.  This pair 
is irreducible since the only occurrences of $a$ are on the top left 
and bottom right.
\item  For   $a<b<c$ in $\Sigma_\bot$, 
$[a,b,c]$ denotes the balanced pair
$\begin{bmatrix} \mathrm{pal}(a)^{-1}\mathrm{pal}(b)\,  c  \\
  c\, \mathrm{pal}(b) \mathrm{pal}(a)^{-1} \end{bmatrix}$.
This pair 
is irreducible since the only occurrences of $c$ are on the top right 
and bottom left.
\end{enumerate}
We consider the following alphabet of irreducible balanced pairs:
\[ \Gamma := \{[a]:a\in \Sigma\} \cup 
  \{[a,b,c]\in (\Sigma_\bot)^3 : a>b>c \text{ or } a<b<c\} \,  .\]

We extend the definition of iterated palindromic closure to a map
$\mathrm{pal}:\Sigma\times \Sigma\rightarrow \Gamma^*$, defined
inductively by $\mathrm{pal}(a,a) := \varepsilon$ and 
\[ \mathrm{pal}(a,b) := \begin{cases}
    \mathrm{pal}(a-1,b)\, [a,b,-1]\, \mathrm{pal}(a-1,b) & (a>b)\\
    \mathrm{pal}(a,b-1) \, [-1,a,b] \, \mathrm{pal}(a,b-1) & (a<b)
  \end{cases}  \]
We use this map to characterise the lifting of $\varphi$ to a morphism
$\widetilde{\varphi} :\Gamma^+\rightarrow \Gamma^+$ as follows:
\[
  \begin{array}{rcll}
  \widetilde{\varphi} ([a]) & := & \begin{cases}
    [0] \, [a+1] &  \text{if $a< k-1$} \\
    [0] \,  & \text{if $a=k-1$} 
  \end{cases}  & \\[0.2ex]
  \widetilde{\varphi} ([a,b,c]) &  :=&  \begin{cases}
[0]\,  [a+1,b+1,c+1] &  \text {if }a< k-1 \\ 
[0]\,    \mathrm{pal}(c+1,b+1) & \text{if }a=k-1 
\end{cases} & \quad (a>b>c)\\
  \widetilde{\varphi} ([a,b,c]) &  :=&  \begin{cases}
[0]\,  [a+1,b+1,c+1] &  \text {if }c< k-1 \\ 
[0]\,    \mathrm{pal}(c+1,b+1) & \text{if }c=k-1
\end{cases}              & \quad (a<b<c)\\
\end{array}
\]
The correctness of the above characterisation can be shown by direct calculation
using the facts that $\varphi(\mathrm{pal}(a)) = \mathrm{pal}(a+1)0^{-1}$ for
$a\in\{0,\ldots,k-2\}$ and
$\varphi(\mathrm{pal}(-1)) = \varepsilon = \mathrm{pal}(-1)$.
For example, for $k-1>a>b>c \geq 0$ we have
\begin{eqnarray*}
  \varphi(a \, \mathrm{pal}(b) \, \mathrm{pal}(c)^{-1}) &=& 0 \, (a+1) \, 
\mathrm{pal}(b+1) \, 0^{-1}\,  0 \, \mathrm{pal}(c+1)^{-1} \\
&=&   0 \, (a+1) \, \mathrm{pal}(b+1)\, \mathrm{pal}(c+1)^{-1} 
\end{eqnarray*}
and
\begin{eqnarray*}
  \varphi(\mathrm{pal}(c)^{-1}  \, \mathrm{pal}(b) \, a) &=&
  0 \, \mathrm{pal}(c+1)^{-1} \, 0 \, 0^{-1} \, \mathrm{pal}(b+1) \,  (a+1)\\ 
  &=& 0 \, \mathrm{pal}(c+1)^{-1}  \mathrm{pal}(b+1)   (a+1) \, .
  \end{eqnarray*}
We conclude that $\widetilde{\varphi}([a,b,c]) =[0]\, [a+1,b+1,c+1]$.
The other cases follow by similar reasoning.

The construction above shows that the balanced-pair algorithm
terminates with coincidence for the $k$-bonacci morphism.
Using Lemma~\ref{lem:binary} we conclude that the $k$-bonacci word is
echoing.  The rest of proof is dedicated to establishing the
Non-Vanishing condition.

The \emph{incidence graph} $\mathcal G_{\widetilde{\varphi}}$ of
$\widetilde{\varphi}$ has set of
vertices the set $\Gamma_0 := \Gamma \setminus \{[a]: a\in\Sigma\}$
comprising the mismatch symbols 
$[a,b,c]$ in $\Gamma$.  There is a directed
edge from $[a,b,c]$ to $[a',b',c']$ if and only if $[a',b',c']$ 
appears in the string $\widetilde{\varphi}([a,b,c])$.

\begin{proposition}
The incidence graph $\mathcal G_{\widetilde{\varphi}}$ has a unique
cycle cover.
\label{prop:unique}
  \end{proposition}
  \begin{proof}
    We specify a cycle cover $\mathcal C$ of
    $\mathcal G_{\widetilde{\varphi}}$ by defining a function
    $s:\Gamma_0\rightarrow \Gamma_0$ that
    determines an outgoing edge from every vertex.  
    Specifically, for $a>b>c$ we have
    \[ s([a,b,c]):= \begin{cases} [a+1,b+1,c+1] & a<k-1\\
        [-1,c+1,b+1] &a=k-1\end{cases} \]
and for $a<b<c$ we have 
\[ s([a,b,c]):= \begin{cases} [a+1,b+1,c+1] & c<k-1\\
    [b+1,a+1,-1] &c=k-1\end{cases} \, . \] Informally, whenever
there is a choice of an outgoing edge from a node $[a,b,c]$ in
$\mathcal G_{\widetilde{\varphi}}$ the function $s$ selects the edge
leading to the central letter of the string
$\widetilde{\varphi}([a,b,c])$.

    It can be seen
    by direct calculation that $\mathcal C$ is a cycle cover.  Indeed,
    $\mathcal C$ is a union of cycles of the following form, where
    $a>b >-1$: 
\begin{align*} [a,b,-1] &  \stackrel{s^{k-a}}{\longmapsto}
  [-1,k-a-1,k+b-a] \stackrel{s^{a-b}}{\longmapsto} [k-b-1,a-b-1,-1]
                            \stackrel{s^{b+1 }}{\longmapsto} [-1,b,a] \\
  & \stackrel{s^{k-a}}{\longmapsto} [k+b-a,k-a-1,-1]
    \stackrel{s^{a-b}}{\longmapsto} [-1,a-b-1,k-b-1]
    \stackrel{s^{b+1}}{\longmapsto} [a,b,-1] \, .
\end{align*}

We now show that $\mathcal C$ is the unique cycle cover of
$\mathcal G_{\widetilde{\varphi}}$.  To this end, recall that a
$\mathcal C$-\emph{alternating cycle} is a bipartite graph
$\mathcal{A}=(U,V,E)$, where $U$ and $V$ are disjoint sets of vertices
of $\mathcal G_{\widetilde{\varphi}}$ with $|U|=|V|\geq 2$,
$E\subseteq U\times V$ is a set of edges of
$\mathcal G_{\widetilde{\varphi}}$, and each vertex of $\mathcal{A}$
is incident to two edges in $E$, of which one lies in $\mathcal C$.
By Berge's Theorem~\cite[Chapter 2]{lovasz1986matching} the uniqueness
of $\mathcal C$ as a cycle cover is equivalent to the non-existence of
a $\mathcal C$-alternating cycle; see Figure~\ref{fig:augment} for a
representation of an alternating cycle.

The nodes in $\mathcal G_{\widetilde{\varphi}}$ that have
outdegree at least two have the form $[k-1,a,b]$ or $[a,b,k-1]$.
Hence for any bipartite subgraph of the form shown in 
Figure~\ref{fig:augment} there are two cases.  The first case is that
$a_0=a_2=\cdots=a_{2m}=k-1$ and $a_1=a_3=\cdots = a_{2m+1}=-1$.
Since the blue edges point to the central factor of an iterated
palindromic closure we have
 $c_1>c_3>\cdots c_{2m+1}>c_1$, which is a contradiction.
The second case is that 
$c_0=c_2=\cdots=c_{2m}=k-1$ and $c_1=c_3=\cdots = c_{2m+1}=-1$. 
Then we have $a_1>a_3>\cdots a_{2m+1}>a_1$, which is again a contradiction.
We conclude that an alternating cycle of the form shown in
Figure~\ref{fig:augment} cannot exist and hence the cycle cover
$\mathcal C$ is unique.
\end{proof}

\begin{center}
\begin{figure}
\begin{tikzpicture}[scale=1, every node/.style={draw=none, minimum size=1cm}, every edge/.append style={->,>=stealth,thick}]
    \node (A1) at (0, 3) {$[a_0,b_0,c_0]$};
    \node (A2) at (0, 1) {$[a_2,b_2,c_2]$};
    \node (A3) at (0, -1) {$\vdots$};
    \node (A4) at (0, -3) {$[a_{2m},b_{2m},c_{2m}]$};

    \node (B1) at (4, 3) {$[a_1,b_1,c_1]$};
    \node (B2) at (4, 1) {$[a_3,b_3,c_3]$};
    \node (B3) at (4, -1) {$\vdots$};
    \node (B4) at (4, -3) {$[a_{2m+1},b_{2m+1},c_{2m+1}]$};

    \draw (A1) edge [color=blue] (B1);
    \draw (A1) edge (B2);

    \draw (A2) edge [color=blue] (B2);
    \draw (A2) edge (B3);

    \draw (A3) edge [color=blue] (B3);
    \draw (A3) edge (B4);
    
    \draw (A4) edge [color=blue] (B4);
    \draw (A4) edge (B1);
  \end{tikzpicture}
\caption{An alternating cycle for a cycle cover $\mathcal C$.  The blue edges lie in  $\mathcal C$ and the 
  black edges are not in $\mathcal C$.  Removing the blue edges from
  $\mathcal C$ and replacing them by 
  the black edges yields a new cycle cover.  On the other hand, if
  $\mathcal C'$ is a cycle cover other than $\mathcal C$ then the
  symmetric difference of the respective sets of edges in $\mathcal C$
and $\mathcal C'$ can be partitioned into alternating paths of the above form}
  \label{fig:augment}
\end{figure}
\end{center}

\subsection{Strong Echoing Condition}
Write $\Gamma_0:=\{a_0,\ldots,a_m\}$ for the set of mismatch symbols
in $\Gamma$.
For $i \in \{0,\ldots,m\}$, define $Q_{n,i}(x) \in \mathbb{Z}[x]$ by
\[ Q_{n,i}(x) := \sum_{k=0}^{\ell-1} (\pi_1(\widetilde{\varphi}^n(a_i))_k - \pi_2(\widetilde{\varphi}^n(a_i))_k)x^{k} 
\, .
  \]
Using Proposition~\ref{prop:unique} we prove the following result:
  
 \begin{proposition}
The vector $\boldsymbol Q_n := (Q_{n,0},\ldots,Q_{n,m}) \in \mathbb
Z[x]^{m+1}$ is such that $\boldsymbol Q_n(\beta^{-1}) \neq \boldsymbol 0$
for all $\beta \in \mathbb C\setminus \{0,1\}$. 
\label{prop:non-vanish}
\end{proposition}
    \begin{proof}
      Without loss of generality, assume that $a_0 \in \Gamma$ is the
      matched pair $(01,10)$.  Then $Q_{0,0}(x)$ is the polynomial
      $-1+x$.  Hence $\boldsymbol Q_0(\beta) \neq \boldsymbol 0$
      for $\beta\neq 1$.
        
      We next describe a recurrence
      $\boldsymbol Q_{n+1} = M_n \, \boldsymbol Q_n$, where
      $M_n \in \mathbb Z[x]^{(m+1)\times (m+1)}$.  The matrix entries
      $(M_n)_{i,j}$ are defined as follows.  Suppose
      $\widetilde{\varphi}(a_i) = a_{j_1} \cdots a_{j_s}$.  Then
      $\widetilde{\varphi}^{n+1}(a_i) = \widetilde{\varphi}^n(a_{j_1})
      \cdots \widetilde{\varphi}^n(a_{j_s})$.  This equation allows us
      to write the polynomial $Q_{n+1,i}(x)$ as a linear combination
      of the polynomials $Q_{n,0}(x),\ldots,Q_{n,m}(x)$.  Specifically,
      we have
      \[ (M_n)_{i,j} := \sum_{k:j_k=j} x^{f(k)}, \quad\text{where }
      f(k) = |\pi_1( \widetilde{\varphi}^n(a_{j_1}) \cdots
      \widetilde{\varphi}^n(a_{j_k-1}))| \, . \]

      Notice that $(M_n)_{i,j}\neq 0$ only if there is an edge from
      $a_i$ to $a_j$ in the incidence graph
      $\mathcal G_{\widetilde{\varphi}}$.  Thus we have
      $\det(M_n) = \prod_{i=0}^m (M_n)_{i,s(i)}$ for the map
      $s : \{0,\ldots,m\}\rightarrow \{0,\ldots,m\}$ that specifies
      the unique cycle cover of $\mathcal G_{\widetilde{\varphi}}$, as
      defined in the proof of Proposition~\ref{prop:unique}.  By
      definition of $s$, the symbol $a_{s(i)}$ has a single occurrence
      in the string $\widetilde{\varphi}(a_i)$, namely as the central factor
      in an iterated palindromic closure.  It follows  that $M_{i,s(i)}$ is
      a monomial for all $i\in\{1,\ldots,m\}$ and hence $\det(M_n)$ is a monomial.
      Thus $M_n(\beta^{-1})$ is non-singular for all $\beta\neq 0$.

      From recurrence $\boldsymbol Q_{n+1} = M_n \, \boldsymbol Q_n$
      and the facts that $\boldsymbol Q_0(\beta) \neq \boldsymbol 0$
      for $\beta\neq 1$ and $\det(M_n)(\beta)\neq 0$ for
      $\beta\neq 0$, we have by induction on $n\in\mathbb N$ that
      $\boldsymbol Q_n(\beta^{-1})\neq \boldsymbol 0$ for all
      $n\in\mathbb N$ and $\beta\in\mathbb C\setminus\{0,1\}$.
      \end{proof}

      We can now prove the main result of this section.
  The proof develops the construction in 
  Sections~\ref{sec:echo} and~\ref{sec:echo2} showing that the 
  Tribonacci word is strongly echoing. 

      \begin{proof}[Proof of Theorem~\ref{thm:kbon}]
  By inspection, the liftting $\widetilde{\varphi}$ of the $k$-Bonacci
  morphism $\varphi$ to the set $\Gamma$ of balanced pairs 
  satisfies the coincidence condition.  Furthermore,
 since $\varphi$ is  primitive, there exists $c_0>0$ 
    such that for all 
    $\gamma,\gamma' \in \Gamma$ and all $n\in 
    \mathbb N$, we have 
\begin{gather}
    |\pi_1(\widetilde{\varphi}^n(\gamma))| \leq c_0\, 
    |\pi_1(\widetilde{\varphi}^n(\gamma'))| \,  .
\label{eq:growth1}
  \end{gather}

  Let $\boldsymbol u \in \Sigma^\omega$ be the 
  $k$-bonacci word and let $\boldsymbol w \in \Gamma^\omega$ be the 
  image of $\boldsymbol u$ under the coding 
  $\iota:\Sigma\rightarrow\Gamma$ given by 
\[ \iota(0) := \begin{bmatrix} 0&1\\ 1&0 \end{bmatrix}
  \quad 
  \iota(1) := \begin{bmatrix} 0&2\\ 2&0 \end{bmatrix}
  \quad  \cdots\quad 
  \iota(k-2) := \begin{bmatrix} 0&k-1\\ k-1&0 \end{bmatrix}
\quad 
  \iota(k-1) := \begin{bmatrix} 0\\0 \end{bmatrix} \, . 
\]
Since $\pi_1\circ \iota = \varphi$, we have
$\pi_1(\boldsymbol w)=\boldsymbol u$.
Moreover since $\sigma\circ \pi_1\circ \iota$ and $\pi_2 \cdot \iota$
coincide on any word on $\Sigma^\omega$, we have 
$\pi_2(\boldsymbol w) =
\sigma(\boldsymbol u)$.

Given $n_0\in \mathbb N$, write 
      \begin{gather}
\widetilde{\varphi}^{n_0}(\boldsymbol w) =   w_0 a_{i_1}
 {w}_{1} a_{i_2}   w_{2} a_{i_3} \cdots 
\label{eq:partition2}
\end{gather}
where $ {w}_0, {w}_1,\ldots \in \Gamma^*$ are sequences of coincidence
symbols and $a_{i_1},a_{i_2,},\ldots \in \Gamma$ are mismatch symbols.
The word $\widetilde{\varphi}^{n_0}(\boldsymbol w)$ is uniformly
recurrent, being the morphic image of the uniformly recurrent word
$\boldsymbol u$.  Hence there is a uniform upper bound on the length
of the factors $ w_i$, depending only on $n_0$.
Let $\varepsilon>0$ be as in the definition of echoing sequence
 (Definition~\ref{def:echoing2}).
Since $\widetilde{\varphi}$ satisfies the coincidence condition, for
sufficiently large $n_0\in\mathbb{N}$ the upper asymptotic density of
mismatch symbols in $\widetilde{\varphi}^{n_0}(\boldsymbol{w})$ is at
most $\varepsilon/c_0$.

The data to show the Strong Echoing property are as follows.
For all $n\in \mathbb{N}$, consider the image of~\eqref{eq:partition2} under the $n$-fold
application of $\widetilde{\varphi}$:
    \begin{gather} \widetilde{\varphi}^{n+n_0}(\boldsymbol{w}) = 
  \widetilde{\varphi}^{n}(w_0) \underbrace{\widetilde{\varphi}^{n}(a_{i_1})}_{I_{n,1}}
  \widetilde{\varphi}^{n}(w_1)  \underbrace{\widetilde{\varphi}^{n}(a_{i_2})}_{I_{n,2}}
  \widetilde{\varphi}^{n}(w_2)
  \underbrace{\widetilde{\varphi}^{n}(a_{i_3})}_{I_{n,2}} \cdots \, .
\label{eq:PART3}
\end{gather}
We define $r_n:=0$ and $s_n:=|\varphi^{n+n_0}(0)|$.  Then
$\pi_1(\widetilde{\varphi}^{n+n_0}(\boldsymbol w)) =\boldsymbol u$ and
$\pi_2(\widetilde{\varphi}^{n+n_0}(\boldsymbol w))
=\sigma^{s_n}(\boldsymbol u)$.  For all $j\in \{1,2,\ldots\}$ we
define $I_{n,j} \subseteq \mathbb{N}$ to be the smallest interval of
positions in
$\boldsymbol u=\pi_1(\widetilde{\varphi}^{n+n_0}(\boldsymbol w))$ that
contains all mismatch symbols occurring in the factor
$\pi_1(\widetilde{\varphi}^{n}(a_{i_j}))$, as shown in~\eqref{eq:PART3}.
 By construction, this choice satisfies the Covering
Property.  To establish the Density Property we note that since the
mismatch symbols in $\widetilde{\varphi}^{n_0}(\boldsymbol w)$ have
upper asymptotic density at most $\varepsilon/c_0$, by~\eqref{eq:growth1} we have
$\mathrm{den}\left(\bigcup_{j=1}^\delta I_{n,j}\right)\leq
\varepsilon$ for all sufficiently large $\delta$.  
Finally, the
uniform upper bound on the length of the factors $w_i$ and the
bound~\eqref{eq:growth1}  imply that 
$d(I_{n,j},I_{n,j+1}) \asymp s_n$ for all $j\geq 1$.  (Note that this
growth bound holds even when the word $w_j$ is empty since
$\widetilde{\varphi}(a_{i_{j+1}})$ starts with a coincidence.)
We have thus established the Expanding-Gaps Property.

It remains to consdier the Non-Vanishing Property.  To this end, for
all $n\in \mathbb N$ and $j\in \{1,2,\ldots\}$, define the polynomial
$P_{n,j}(x) \in \mathbb{Z}[x]$ by
\begin{gather} P_{n,j}(x) := \sum_{i \in I_{n,j}} (u_{i+s_n}-u_{i+r_n})x^{i} \,
.
\end{gather}
Given~$\beta\in\mathbb C$ with $|\beta|>1$, we must show that there
exist distinct indices $j_0,j_1$ such that
$P_{n,j_0}(\beta^{-1})\neq 0$ and $P_{n,j_1}(\beta^{-1})\neq 0$ for
infinitely many $n\in\mathbb N$.  By
Proposition~\ref{prop:non-vanish}, for all $n\in\mathbb N$ there
exists $i\in \{0,\ldots,m\}$ such that $Q_{n,i}(\beta^{-1})\neq 0$.
Hence there exists $i^*$ such that $Q_{n,i^*}(\beta^{-1})\neq 0$ for
infinitely many $n\in\mathbb N$.  There are arbitrarily large choices
of $n_0$ such that $i^*$ occurs 
infinitely often in the list $i_1,i_2,i_3,\ldots$ in~\eqref{eq:partition2}.  For
such an $n_0$  there exist
$j_0<j_1$ such that $i_{j_1}=i_{j_2}=i^*$.  Then
$Q_{n,i_{j_0}}(\beta^{-1})\neq 0$ and
$Q_{n,i_{j_1}}(\beta^{-1})\neq 0$ for infinitely many $n$.  But for
all $j \in \{1,2,\ldots\}$, $P_{n,j}$ is the product of $Q_{n,i_j}$
and a monomial in $x$.  Hence $P_{n,{j_0}}(\beta^{-1})\neq 0$ and
$P_{n,{j_1}}(\beta^{-1})\neq 0$ for infinitely many $n$.
  \end{proof}

\section{Conclusion}
Resolving the Pisot conjecture is an important research objective in
dynamical systems~\cite{Akiyama2015,MercatAkiyama2020}.  Our main
results show that progress on this question has implications for the
version of Cobham's conjecture concerning expansions in an algebraic
base.  A natural direction for further research is to determine
classes of morphisms that lead to strongly echoing words, rather than
merely echoing words (for which we would obtain transcendence outright
by applying Theorem~\ref{thm:main2}, rather than the weaker
rational-transcendence dichotomy that is obtained from
Theorem~\ref{thm:main}).  In the present work we have obtained such a
result for the class of $k$-bonacci morphisms.  One can ask whether
the strong echoing property holds for the sub-class of Arnoux-Rauzy words
whose $S$-adic expansion is periodic (see~\cite[Chapter
7]{pytheas2002substitutions} for the notion of $S$-adic expansion).  In this regard, we note that a number
whose expansion in an integer base is an Arnoux-Rauzy word is
necessarily transcendental~\cite{BFZ,FM}.  On the other
hand,~\cite[Section 3.1]{PK} gives an example of the Episturmian word
$\boldsymbol u$ and an algebraic base $\beta$ such that the
Non-Vanishing condition fails and the number
$\llbracket\boldsymbol{u}\rrbracket_{\beta}$ belongs to $\mathbb Q(\beta)$.

\bibliography{bibliography}

\begin{thebibliography}{10}

\bibitem{AB1}
B.~Adamczewski and Y.~Bugeaud.
\newblock On the complexity of algebraic numbers {I}. expansions in integer
  bases.
\newblock {\em Annals of Mathematics}, 165:547--565, 2005.

\bibitem{AB2}
B.~Adamczewski and Y.~Bugeaud.
\newblock Dynamics for $\beta$-shifts and {D}iophantine approximation.
\newblock {\em Ergodic Theory and Dynamical Systems}, 27:1695 -- 1711, 2007.

\bibitem{ABL}
B.~Adamczewski, Y.~Bugeaud, and F.~Luca.
\newblock Sur la complexit{\'e} des nombres alg{\'e}briques.
\newblock {\em Comptes Rendus Mathematique}, 339:11--14, 2004.

\bibitem{ACG}
B.~Adamczewski, J.~Cassaigne, and M.~Le~Gonidec.
\newblock On the computational complexity of algebraic numbers: the
  {H}artmanis--{S}tearns problem revisited.
\newblock {\em Transactions of the American Mathematical Society},
  373(5):3085--3115, 2020.

\bibitem{AF}
B.~Adamczewski and C.~Faverjon.
\newblock Mahler's method in several variables and finite automata.
\newblock {\em to appear in Ann. of Math., 66 pp.}, 2025.

\bibitem{Akiyama2015}
S.~Akiyama, M.~Barge, V.~Berth\'{e}, J.-Y. Lee, and A.~Siegel.
\newblock On the {P}isot substitution conjecture.
\newblock In Johannes Kellendonk, Daniel Lenz, and Jean Savinien, editors, {\em
  Mathematics of Aperiodic Order}, volume 309 of {\em Progress in Mathematics},
  pages 33--72. Birkhäuser, Basel, 2015.
\newblock \href {https://doi.org/10.1007/978-3-0348-0903-0_2}
  {\path{doi:10.1007/978-3-0348-0903-0_2}}.

\bibitem{AR91}
P.~Arnoux and G.~Rauzy.
\newblock Repr{\'e}sentation g{\'e}om{\'e}trique de suites de complexit{\'e} $2
  n+ 1$.
\newblock {\em Bulletin de la Soci{\'e}t{\'e} math{\'e}matique de France},
  119(2):199--215, 1991.

\bibitem{Barge2002}
M.~Barge and B.~Diamond.
\newblock Coincidence for substitutions of {P}isot type.
\newblock {\em Bulletin de la Société Mathématique de France},
  130(4):619--626, 2002.

\bibitem{pytheas2002substitutions}
V.~Berth\'{e}, S.~Ferenczi, C.~Mauduit, and A.~Siegel, editors.
\newblock {\em Substitutions in Dynamics, Arithmetics and Combinatorics},
  volume 1794 of {\em Lecture Notes in Mathematics}.
\newblock Springer-Verlag, 2002.
\newblock \href {https://doi.org/10.1007/978-3-540-45576-7}
  {\path{doi:10.1007/978-3-540-45576-7}}.

\bibitem{BFZ}
V.~Berthe, S.~Ferenczi, and L.~Q. Zamboni.
\newblock Interactions between dynamics, arithmetics and combinatorics: The
  {G}ood, the {B}ad, and the {U}gly.
\newblock {\em Algebraic and Topological Dynamics}, 385, 2005.

\bibitem{Cob2}
A.~Cobham.
\newblock Uniform tag seqences.
\newblock {\em Math. Syst. Theory}, 6(3):164--192, 1972.

\bibitem{Luca97a}
A.~de~Luca.
\newblock Sturmian words: Structure, combinatorics, and their arithmetics.
\newblock {\em Theor. Comput. Sci.}, 183(1):45--82, 1997.
\newblock \href {https://doi.org/10.1016/S0304-3975(96)00310-6}
  {\path{doi:10.1016/S0304-3975(96)00310-6}}.

\bibitem{FM}
S.~Ferenczi and C.~Mauduit.
\newblock Transcendence of numbers with a low complexity expansion.
\newblock {\em Journal of Number Theory}, 67(2):146--161, 1997.

\bibitem{Glasner2003}
E.~Glasner.
\newblock {\em Ergodic Theory via Joinings}, volume 101 of {\em Mathematical
  Surveys and Monographs}.
\newblock American Mathematical Society, Providence, RI, 2003.
\newblock \href {https://doi.org/10.1090/surv/101}
  {\path{doi:10.1090/surv/101}}.

\bibitem{Hollander2003TwosymbolPS}
M.~Hollander and B.~Solomyak.
\newblock Two-symbol {P}isot substitutions have pure discrete spectrum.
\newblock {\em Ergodic Theory and Dynamical Systems}, 23:533 -- 540, 2003.

\bibitem{host1989spectral}
B.~Host and F.~Parreau.
\newblock Spectral theory of dynamical systems.
\newblock {\em Ergodic Theory and Dynamical Systems}, 9(1):91–116, 1989.
\newblock \href {https://doi.org/10.1017/S0143385700005031}
  {\path{doi:10.1017/S0143385700005031}}.

\bibitem{PK}
P.~Kebis.
\newblock Transcendence of numbers related to episturmian words.
\newblock Master's thesis, University of Oxford, 2023.
\newblock Available at: \url{https://ora.ox.ac.uk}.

\bibitem{KebisLOS024}
P.~Kebis, F.~Luca, J.~Ouaknine, A.~Scoones, and J.~Worrell.
\newblock On transcendence of numbers related to {S}turmian and
  {A}rnoux-{R}auzy words.
\newblock In {\em 51st International Colloquium on Automata, Languages, and
  Programming, {ICALP} 2024}, volume 297 of {\em LIPIcs}, pages 144:1--144:15.
  Schloss Dagstuhl - Leibniz-Zentrum f{\"{u}}r Informatik, 2024.
\newblock \href {https://doi.org/10.4230/LIPICS.ICALP.2024.144}
  {\path{doi:10.4230/LIPICS.ICALP.2024.144}}.

\bibitem{LEN97}
H.~W. Lenstra~Jr.
\newblock Finding small degree factors of lacunary polynomials.
\newblock {\em Number theory in progress}, 1:267--276, 1999.

\bibitem{Livshits1987}
A.~N. Livshits.
\newblock On the spectra of adic transformations of markov compacta.
\newblock {\em Uspekhi Matematicheskikh Nauk}, 42(3):189--190, 1987.
\newblock In Russian.

\bibitem{Livshits1992}
A.~N. Livshits.
\newblock Some examples of adic transformations and automorphisms of
  substitutions.
\newblock {\em Selecta Mathematica Sovietica}, 11(1):83--104, 1992.

\bibitem{lovasz1986matching}
L.~Lovász and M.~D. Plummer.
\newblock {\em Matching Theory}, volume~29 of {\em Annals of Discrete
  Mathematics}.
\newblock North-Holland, Amsterdam, 1986.

\bibitem{LOW}
F.~Luca, J.~Ouaknine, and J.~Worrell.
\newblock On the transcendence of a series related to {S}turmian words, 2022.
\newblock To appear, Annali della Scuola Normale Superiore di Pisa.
\newblock \href {https://arxiv.org/abs/2204.08268} {\path{arXiv:2204.08268}}.

\bibitem{MercatAkiyama2020}
P.~Mercat and S.~Akiyama.
\newblock Yet another characterization of the {P}isot substitution conjecture.
\newblock In {\em Substitution and Tiling Dynamics: Introduction to
  Self-inducing Structures}, volume 2273 of {\em Lecture Notes in Mathematics},
  pages 397--448. Springer, 2020.
\newblock URL:
  \url{https://link.springer.com/chapter/10.1007/978-3-030-57666-0_8}, \href
  {https://doi.org/10.1007/978-3-030-57666-0_8}
  {\path{doi:10.1007/978-3-030-57666-0_8}}.

\bibitem{Queffelec2010}
M.~Queffélec.
\newblock {\em Substitution Dynamical Systems -- Spectral Analysis}, volume
  1294 of {\em Lecture Notes in Mathematics}.
\newblock Springer, 2 edition, 2010.
\newblock URL: \url{https://link.springer.com/book/10.1007/978-3-642-11212-6},
  \href {https://doi.org/10.1007/978-3-642-11212-6}
  {\path{doi:10.1007/978-3-642-11212-6}}.

\bibitem{Rauzy1982}
G.~Rauzy.
\newblock Nombres algébriques et substitutions.
\newblock {\em Bulletin de la Société Mathématique de France}, 110:147--178,
  1982.

\bibitem{Schlickewei76}
H.~P. Schlickewei.
\newblock Die p-adische verallgemeinerung des {S}atzes von
  {T}hue-{S}iegel-{R}oth-{S}chmidt.
\newblock {\em Journal für die reine und angewandte Mathematik},
  1976(288):86--105, 1976.

\bibitem{Wald2000}
M.~Waldschmidt.
\newblock {\em {D}iophantine Approximation on Linear Algebraic Groups}, volume
  326.
\newblock 2000.
\newblock \href {https://doi.org/10.1007/978-3-662-11569-5}
  {\path{doi:10.1007/978-3-662-11569-5}}.

\end{thebibliography}
\end{document}